\documentclass[a4paper,10pt]{amsart}

\usepackage[utf8]{inputenc} 

\usepackage{amsmath}
\usepackage{amstext}
\usepackage{amsfonts}
\usepackage{amsthm}
\usepackage{amssymb}
\usepackage{enumitem}   

\usepackage{graphicx}
\usepackage{dsfont}

\usepackage{hyperref}
\hypersetup{colorlinks}

\usepackage{xcolor}

\numberwithin{equation}{section}

\newcommand{\R}{{\mathbf{R}}}

\newcommand{\E}{{\mathbf{E}}}
\newcommand{\N}{{\mathbf{N}}}
\newcommand{\Q}{{\mathbf{Q}}}
\newcommand{\D}{{\mathcal{D}}}
\newcommand{\F}{{\mathcal{F}}} 
\renewcommand{\P}{{\mathbf{P}}} 
\newcommand{\dom}{{\mathrm{dom}}}
\newcommand{\B}{{\mathcal{B}}}
\newcommand{\Pow}{{\mathcal{P}}}
\newcommand{\A}{{\mathcal{A}}}
\renewcommand{\L}{{\mathcal{L}}}
\newcommand{\NN}{{\mathcal{N}}}

\newcommand{\diff}[1]{\,\mathrm{d}#1}
\newcommand{\inner}[3][]{\langle#2,#3\rangle_{#1}}
\newcommand{\ska}[3][]{( #2,#3 )_{#1}}

\newcommand{\ee}{\mathrm{e}}

\newcommand{\triple}{{\vert\kern-0.25ex\vert\kern-0.25ex\vert}}

\newcommand{\tn}{\xi}

\theoremstyle{plain}

\newtheorem{definition}{Definition}[section]
\newtheorem{theorem}[definition]{Theorem}
\newtheorem{lemma}[definition]{Lemma}

\newtheorem{prop}[definition]{Proposition}
\newtheorem{assumption}[definition]{Assumption}

\theoremstyle{definition}
\newtheorem{remark}[definition]{Remark}

\newtheorem{example}[definition]{Example}

\begin{document}

\title[On a randomized backward Euler method]
{On a randomized backward Euler method for\\ 
nonlinear evolution equations with\\ time-irregular coefficients}

\author[M.~Eisenmann]{Monika Eisenmann}
\address{Monika Eisenmann\\
Technische Universit\"at Berlin\\
Institut f\"ur Mathematik, Secr. MA 5-3\\
Stra\ss e des 17.~Juni 136\\
DE-10623 Berlin\\
Germany}
\email{meisenma@math.tu-berlin.de}

\author[M.~Kov\'acs]{Mih\'aly Kov\'acs}
\address{Mih\'aly Kov\'acs\\
Department of Mathematical Sciences\\
Chalmers University of Technology and University of Gothenburg\\
SE-412 96 Gothenburg\\
Sweden}
\email{mihaly@chalmers.se}

\author[R.~Kruse]{Raphael Kruse}
\address{Raphael Kruse\\
Technische Universit\"at Berlin\\
Institut f\"ur Mathematik, Secr. MA 5-3\\
Stra\ss e des 17.~Juni 136\\
DE-10623 Berlin\\
Germany}
\email{kruse@math.tu-berlin.de}

\author[S.~Larsson]{Stig Larsson}
\address{Stig Larsson\\
Department of Mathematical Sciences\\
Chalmers University of Technology and University of Gothenburg\\
SE-412 96 Gothenburg\\
Sweden}
\email{stig@chalmers.se}


\begin{abstract}
  In this paper we introduce a randomized version of the
  backward Euler method,
  that is applicable to stiff ordinary differential equations and
  nonlinear evolution equations with time-irregular
  coefficients. In the finite-dimensional  case,
  we consider Carath\'eodory type functions satisfying
  a one-sided Lipschitz condition. 
  After investigating the well-posedness and the stability properties
  of the randomized scheme, we prove the convergence to the exact solution
  with a rate of $0.5$ in the root-mean-square norm assuming only that
  the coefficient function is square
  integrable with respect to the temporal parameter.
  
  These results are then extended to the numerical solution of
  infinite-dimen\-sional  evolution equations under monotonicity and
  Lipschitz conditions. Here we consider a combination of the randomized
  backward Euler scheme with a Galerkin finite element method. 
  We obtain error estimates that correspond to the regularity of the exact 
  solution. 
  The practicability of the randomized scheme is also illustrated through
  several numerical experiments.
\end{abstract}

\maketitle

\section{Introduction}
\label{sec:intro}

The aim of this paper is to introduce a new numerical scheme to 
approximate the solution of an ordinary differential equation (ODE) of
Carath\'{e}odory type
\begin{align} 
  \label{eq1:ODE}
	\begin{split}
		\begin{cases}
			\dot{u}(t) = f(t, u(t)), \quad \text{ for almost all } t\in (0,T],\\
			u(0) 		= u_0, 
		\end{cases}
	\end{split}
\end{align}
for $T \in (0, \infty)$, and of a non-autonomous evolution equation 
\begin{align} \label{eq1:PDE}
	\begin{split}
		\begin{cases}
			\dot{u}(t) + \A(t) u(t) = f(t), \quad \text{for almost all } t \in 
			(0,T],\\
			u(0) = u_0,    
		\end{cases}
	\end{split}
\end{align}
where $\A \colon [0,T] \times V \to V^\ast$ is a strongly monotone and 
Lipschitz continuous operator with respect to the second argument
that is defined on a Gelfand triple $V \hookrightarrow H \cong
H^\ast \hookrightarrow V^\ast$ for real Hilbert spaces $V$ and $H$.

We focus on the particular difficulty that the mappings $f$ and
$\A$ are irregular with respect to the temporal parameter. More precisely, we
do not impose any continuity conditions but only certain integrability
requirements with respect to $t$. For a concise description of the general
settings we refer to Sections~\ref{sec:ODE} and~\ref{sec:nonlinPDE},
respectively. In particular, a precise 
statement of all conditions is given in Assumption~\ref{as:ODEf} for
\eqref{eq1:ODE} and in Assumption~\ref{as:nonlinPDE} for  
\eqref{eq1:PDE}. To develop the idea of our scheme we mostly focus on the ODE
problem \eqref{eq1:ODE} in this introduction. The derivation of the
numerical scheme for the evolution equation \eqref{eq1:PDE} follows 
analogously and will be introduced in detail in Section~\ref{sec:nonlinPDE}. 

When considering a right-hand side $f$ that is only integrable, every
deterministic algorithm can be "fooled" if it only uses information
provided by point evaluations on prescribed (deterministic) points.
One can easily construct suitable fooling functions for general classes of 
deterministic algorithms, for instance, based on adaptive strategies. 
In Sections~\ref{sec:numexpODE} and \ref{sec:PDEnum} we give examples 
of such fooling functions and investigate the numerical behavior. Further,
we refer to the vast literature on the information-based complexity theory 
(IBC), which applies similar techniques to derive lower bounds for the error 
of certain classes of deterministic and randomized numerical algorithms. For 
instance, see \cite{novak1988, traub1988} for a general introduction into 
IBC and \cite{heinrich2008, kacewicz1987, kacewicz2006} for applications 
to the numerical solution of initial value problems.

One way to construct numerical methods for the solution of
initial value problems with time-irregular coefficients consists 
of allowing the algorithm to use additional information of the
right-hand side $f$ as, for example, integrals of the form
\begin{align}
  \label{eq1:linfunc}
  f^n(x) := \frac{1}{t_n - t_{n-1}}  \int_{t_{n-1}}^{t_n} f(s, x)
  \diff{s},\quad \text{for } x \in \R^d. 
\end{align}
This approach is often found in the existence theory of
ODEs and PDEs when a numerical method is
used to construct analytical solutions to the initial value problems
\eqref{eq1:ODE} and \eqref{eq1:PDE} under minimal regularity assumptions.
The complexity of such methods has also been
studied in \cite{kacewicz1987} (and the references therein) for the numerical 
solution of ODEs. It is also the state-of-the-art
method in many recent papers for the
numerical solution of evolution equations of the form \eqref{eq1:PDE}. For
example, we refer to \cite{baiocchi1983, emmrich2009,
houZhu2006,meidnerVexler2017}.  

However, it is rarely discussed how a quantity such as $f^n(x)$
in \eqref{eq1:linfunc} is obtained in practice. Strictly speaking,
since the computation of $f^n(x)$ often requires the application of further
numerical methods such as quadrature rules, algorithms relying on integrals
such as \eqref{eq1:linfunc} are, in general,
not fully discrete solvers yet. More importantly, classical quadrature rules
for the approximation of 
$f^n(x)$ are again based on deterministic point evaluations of $f$ and may
therefore be ``fooled''.

Instead of using linear functionals such as \eqref{eq1:linfunc}
we propose the following \emph{randomized} version of the
backward Euler method. For $N \in \N$, a step size $k = \frac{T}{N}$, and a
temporal grid $0 = t_0  <t_1< \dots <t_N = T$ with $t_n = nk$ for $n \in \{
0,\dots,N \}$, the randomized scheme for the numerical solution
\eqref{eq1:ODE} is then given by 
\begin{align}	\label{eq1:RandBackEuler}
	\begin{split}
		\begin{cases}
			U^n = U^{n-1} + k f( \tn_n, U^n ),\quad \text{ for } 
			n \in \{ 1,\ldots,N \},\\ 
			U^0 = u_0,
		\end{cases}
	\end{split}
\end{align}
where $\tn_n$ is a uniformly distributed random variable with values in the 
interval $[t_{n-1},t_n]$. Note that we evaluate the 
right-hand side at random points between the grid points. 
Since the evaluation points vary every time the algorithm is called,
it is not possible to construct a fooling function as described above.

We will prove in Theorem~\ref{th:convRandEuler} that the numerical 
solution $U^n$ from \eqref{eq1:RandBackEuler} converges with (strong) 
order $\frac{1}{2}$ to the exact solution $u$ of \eqref{eq1:ODE}, even if $f$ 
is only square integrable with respect to time. Due to the results in 
\cite{heinrich2008} this convergence rate is \emph{optimal} in the sense 
that there exists no deterministic or randomized algorithm based on finitely 
many point evaluations of $f$ with a higher convergence rate within the 
class of all initial value problems satisfying Assumption~\ref{as:ODEf}.

The error analysis is based on the observation that the randomized 
scheme \eqref{eq1:RandBackEuler} is a hybrid of an implicit Runge--Kutta 
method and a Monte Carlo quadrature rule. In fact, if the ODE \eqref{eq1:ODE} 
is actually autonomous, that is, $f$ does not depend on $t$, then we recover 
the classical backward Euler method. On the other hand, if $f$ is independent 
of the state variable $u$, then the ODE \eqref{eq1:ODE} reduces to an 
integration problem and the randomized scheme \eqref{eq1:RandBackEuler} 
is the randomized Riemann sum for the approximation of $u_0 + \int_0^{t_n}
f(s) \diff{s}$ given by
\begin{align*}
  U^n = u_0 + k \sum_{j = 1}^n f(\tn_j), \quad \text{for } n \in
  \{1,\ldots,N\}.
\end{align*}
Observe that a randomized Riemann sum is a particular case of stratified
sampling from Monte Carlo integration. 
It was introduced in \cite{haber1966}, \cite{haber1967} together with further, 
higher order, quadrature rules.
Our error analysis of the randomized scheme \eqref{eq1:RandBackEuler}
combines techniques for the analysis of both time-stepping schemes and Monte 
Carlo integration. 
In particular, since we are interested in the discretization of evolution
equations in later sections, we apply techniques for the numerical analysis of
stiff ODEs developed in \cite{hairer2010} and for stochastic ODEs in
\cite{andersson2017}.  

Before we give a more detailed account of the remainder of this paper, let
us emphasize a few practical advantages of the randomized scheme
\eqref{eq1:RandBackEuler}:

\begin{enumerate}
  \item The implementation of the randomized
    scheme \eqref{eq1:RandBackEuler} is as difficult as
    for the classical backward Euler method in terms of the requirements of 
    solving a nonlinear system of equations. On the other hand,
    the scheme \eqref{eq1:RandBackEuler} does not require 
    integrals such as $f^n(x)$ if the right-hand side is
    time-irregular.

  \item The same is true for the computational effort. 
    Compared to the classical backward Euler method, the randomized
    scheme \eqref{eq1:RandBackEuler} only requires in each step 
    the additional simulation of a single scalar-valued random variable.
    In general, the resulting additional computational effort is negligible
    compared to the solution of a potentially high-dimensional nonlinear system
    of equations. More importantly, due to the randomization we avoid
    the potentially costly computation of the integrals $f^n(x)$.

  \item In contrast to
    every deterministic method based on point evaluations of $f$,
    the randomized scheme \eqref{eq1:RandBackEuler} is independent of the
    particular representation of an integrable function. To be more precise,
    let $g_1$ and $g_2$ be two representations of the same equivalence class $g
    \in L^2(0,T)$. Then, it follows that $g_1(\tn_n) = g_2(\tn_n)$ with
    probability one, since $g_1 = g_2$ almost everywhere. 
\end{enumerate}

We remark that the last item is only valid as long as the random variable
$\tn_n$ is indeed uniformly distributed in $[t_{n-1},t_n]$. 
In practice, however, one usually applies a pseudo-random number generator
which only draws values from the set of floating point numbers. Since
this is a null set with respect to the Lebesgue 
measure, the argument given above is no longer valid. Of course, this problem
affects any algorithm that uses the floating point arithmetic. Nevertheless, a
randomized algorithm is often more robust regarding the particular choice of
the representation of an equivalence class in $L^2(0,T)$ and, hence, more 
user-friendly. For instance, the mapping $(0,T) \ni t \mapsto (T -t )^{-
\frac{1}{3}}$ causes problems for the classical backward
Euler method as it will
evaluate the mapping in the singularity at $t = T$. This problem does not occur
for the randomized backward Euler method with probability one.

Let us also mention that randomized algorithms for the numerical solution of
initial value problems have already been studied in the literature. 
In the ODE case, the complexity and optimality of such algorithms is considered
in \cite{daun2011,heinrich2008, kacewicz2006} under various degrees of
smoothness of $f$. The time-irregular case studied in the present paper was 
first investigated in \cite{stengle1990, stengle1995}. See also
\cite{jentzen2009d,kruse2017a} for a more recent exposition of explicit
randomized schemes. 

The present paper extends the earlier results in several directions. In order 
to deal with possibly stiff ODEs we consider a randomized
version of the backward Euler method and prove its well-posedness and stability
under a one-sided Lipschitz condition. In addition, we require only local 
Lipschitz conditions with respect to the state variable in order to obtain 
estimates on the local truncation error, thereby extending results from
\cite{kruse2017a}. We also avoid any (local) boundedness condition on $f$ 
as, for example, in \cite{daun2011, jentzen2009d}. 

The stability properties also qualify the randomized backward Euler method
as a suitable temporal integrator for non-autonomous evolution equations
with time-irregular coefficients. To the best of our knowledge, there is
no work found in the literature that applies a randomized algorithm
to the numerical solution of evolution equations of the form \eqref{eq1:PDE}. 
Instead, the standard approach in the time-irregular case
relies on the availability of suitable integrals of the right-hand
side as in \eqref{eq1:linfunc}. In particular, we mention \cite{emmrich2009,
houZhu2006}. Further results on optimal rates under minimal regularity
assumptions for linear parabolic PDEs can be found, e.g., in 
\cite{baiocchi1983, chrysafinos2002, hackbusch1981}. For semilinear parabolic
problems optimal error estimates are also found in \cite{meidnerVexler2017},
where a discontinuous Galerkin method in time and space is considered. 

This paper is organized as follows. In Section~\ref{sec:prelim}, we shortly 
introduce the notation and recapture some important concepts of 
stochastic analysis that are relevant for this paper. 
In the following Section~\ref{sec:ODE}, we state the assumptions 
imposed on the ODE \eqref{eq1:ODE}. We also discuss existence and 
regularity of the solution. 
In Section~\ref{sec:randEul}, we then prove the 
well-posedness and convergence of the randomized backward Euler method in the
root-mean-square sense. The ODE part of this paper is completed in
Section~\ref{sec:numexpODE} by examining a numerical example.

In Section~\ref{sec:nonlinPDE}, we introduce the setting for the 
irregular non-autonomous evolution equation \eqref{eq1:PDE} 
that we consider in the second part of
this paper. Under some additional regularity assumptions on the exact solution,
we prove the convergence of a fully discrete method that
combines the randomized backward Euler scheme with a Galerkin finite element
method. The additional regularity assumption is then discussed in more detail 
in Section~\ref{sec:PDEreg}. 
In particular, it is shown that the regularity condition is fulfilled 
for rather general classes of linear and semilinear parabolic PDEs.
Finally in Section~\ref{sec:PDEnum}, we demonstrate that this new randomized 
method can be applied to evolution equations. To this end, we present a 
numerical example which is based
on the finite element software package FEniCS \cite{fenics2012}.

\section{Preliminaries}
\label{sec:prelim}

In this section, we explain the necessary tools from probability theory and 
recall some important inequalities that are needed. First, we start 
by fixing the notation used in this paper.

We denote the set of all positive integers by $\N$ and the set of all real 
numbers by $\R$. In $\R^d$, $d\geq 1$, we denote the Euclidean norm by $| 
\cdot |$ which coincides with the absolute value of a real number for $d=1$. 
The standard inner product in $\R^d$ is denoted by $\ska[]{\cdot}{\cdot}$. For 
a ball of radius $r$ with center $x \in \R^d$ we write $B_r(x) \subseteq 
\R^d$. 

In the following, we will consider different spaces of functions with values 
in general Hilbert spaces. To this end, let $(H,\ska[H]{\cdot}{\cdot}, \|\cdot 
\|_H)$ be a real Hilbert space and $T>0$. 
We will denote the space of continuous functions on $[0,T]$ with 
values in $H$ by $C([0,T]; H)$ where the norm is given by
\begin{align*}
	 \|f \|_{C([0,T];H)} = \sup_{t\in [0,T]} \|f(t)\|_H.
\end{align*}
It will also be important to consider functions which are a little more 
regular. For $0< \gamma <1$ we denote the space of H\"older continuous 
functions by $C^{\gamma}([0,T]; H)$ with norm given by
\begin{align*}
	\|f \|_{C^{\gamma}([0,T];H)} 
	= \sup_{t\in [0,T]} \|f(t)\|_H 
		+ \sup_{\substack{s,t\in [0,T] \\ s \neq t }} \frac{\|f(s) - f(t)\|_H}{|s 
		- t|^{\gamma}}.
\end{align*}
For $p \in[1,\infty)$, we introduce the Bochner--Lebesgue space 
\begin{align*}
	L^p(0,T;H) = \left\{ u:[0,T] \to H \, : \,u \text{ is strongly measurable  
	and 
	}\|u\|_{L^p(0,T;H)} < \infty \right\}
\end{align*} 
where the norm is given by
\begin{align*}
	\|u\|_{L^p(0,T;H)}^p 
	= \int_{0}^{T} \|u(t)\|_{H}^p \diff{t}. 
\end{align*} 
In the case $H = \R$ we write $L^p(0,T)$.

The space of linear bounded operators from $H$ to a Banach space 
$(U, \|\cdot \|_U)$ is denoted by $\L(H,U)$ and in the case of $U=H$ we write 
$\L(H)$. The norm of this space is the usual operator norm 
given by
\begin{align*}
	\|A \|_{\L(H,U)} = \sup_{v \in H, \|v\|_H = 1 }  \|Av \|_{U}.
\end{align*}

Since we are interested in a randomized scheme, we will briefly recall the 
most important probabilistic concepts needed in this paper. 
To this end, we consider a probability space $(\Omega, \F, \P)$ which consists 
of a measurable space $(\Omega, \F)$ together with a finite measure $\P$ such 
that $\P (A) \in [0,1]$ for every $A \in \F$ and $\P(\Omega) = 1$. A mapping 
$X\colon \Omega \to H$ is called a random variable if it is measurable with 
respect to the $\sigma$-algebra $\F$ and the Borel $\sigma$-algebra $\B(H)$ in 
$H$, i.e., for every $B\in \B(H)$ 
\begin{align*}
	X^{-1}(B) = \{ \omega \in \Omega\, : \,X(\omega) \in B \}
\end{align*}
is an element of $\F$. The integral of a random variable $X$ with respect to 
the measure $\P$ is often denoted by
\begin{align*}
	\E[ X ] = \int_{\Omega} X(\omega) \diff{\P(\omega)}.
\end{align*}
The space of $\F$-measurable random variables $X$ such that $\E [ \|X 
\|_H]$ is finite is denoted by $L^1(\Omega, \F, \P; H)$.

For our purposes it is important to consider the space $L^2(\Omega, \F, \P; 
H)$ of square integrable $\F$-measurable random variables. This space is 
often abbreviated by $L^2(\Omega; H)$ if it is clear from the context which 
$\sigma$-algebra $\F$ and measure $\P$ is used.
The space is endowed with the norm 
\begin{align*}
	\|X \|^2_{L^2(\Omega; H)} 
	= \int_{\Omega} \| X(\omega) \|_H^2 \diff{\P(\omega)}
	= \E [\|X \|_H^2], \quad X \in L^2(\Omega;H). 
\end{align*}
Equipped with this norm and inner product 
\begin{align*}
	\ska[L^2(\Omega;H)]{X_1}{X_2} 
	= \int_{\Omega} \ska[H]{X_1(\omega)}{X_2(\omega)} \diff{\P(\omega)}, \quad 
	X_1, 	X_2 \in L^2(\Omega;H), 
\end{align*}
the space $L^2(\Omega;H)$ is a Hilbert space. 

A further important concept is the independence of events $(A_n)_{n \in \N} 
\subset \F$. 
We call the events $(A_n)_{n \in \N}$ independent if for every finite subset 
$I \subset \N$
\begin{align*}
	\P \Big( \bigcap_{n\in I} A_n \Big) = \prod_{n\in I} \P(A_n)
\end{align*}
holds.
This concept can be transferred to families  $(\F_n)_{n\in \N}$ of 
$\sigma$-algebras. Such a family is called independent if for every finite 
subset $I \subset \N$ it follows that every choice of events $(A_n)_{n\in 
I}$ with $A_n \in \F_n$ are independent.
Similarly, a family of $H$-valued random variables $(X_n)_{n\in \N}$ is 
called independent if the generated $\sigma$-algebras
\begin{align*}
	\sigma(X_n) = \{ X_n^{-1}(B) \, : \,B \in \B(H) \}
\end{align*}
are independent.

A family $(\F_n)_{n\in \N}$ of $\sigma$-algebras is called a filtration if for 
every $n\in \N$ the $\sigma$-algebra $\F_n$ is a subset of $\F$  and $\F_n 
\subset \F_m$ holds for $n \leq m$. 
Thus a random variable $X$ can be measurable with respect to $\F_m$ but 
not with respect to $\F_n$ for $n<m$. In some of the arguments in this paper 
it will be important to project an $\F_m$-measurable random variable to a 
smaller $\sigma$-algebra $\F_n$. 
To this end, we introduce the conditional expectation of $X$ with respect to 
$\F_n$: For a random variable $X \in L^1(\Omega, \F_m, \P; H)$ we introduce 
the $\F_n$-measurable random variable $\E[X|\F_n] \colon \Omega \to H$ 
which fulfills
\begin{align*}
	\E [X \mathds{1}_{A}] = \E [ \E[X|\F_n] \mathds{1}_{A} ]
\end{align*}
for every $A \in \F_m$ where $\mathds{1}_{A}$ is the characteristic function 
with respect to $A$.
The random variable $\E[X|\F_n]$ is uniquely determined by these postulations. 
An important property of the conditional expectation of $X \in L^1(\Omega, \F, 
\P; 
H)$ is the 
tower property which states that for two $\sigma$-algebras $\F_n$ and $\F_m$ 
of the filtration $(\F_n)_{n\in \N}$ with $\F_n \subseteq 
\F_m$ we obtain that 
\begin{align*}
	\E [ \E [X|\F_{n} ] | \F_{m}] 
	= \E [ \E [X|\F_{m} ] | \F_{n}] 
	= \E [ X| \F_{n}].
\end{align*}
In particular, if $X$ is already measurable with respect to $\F_n$ then $ \E [ 
X| \F_{n}] = X$ holds. If $\sigma(X)$ is independent of $\F_n$ we obtain that 
$\E [X | \F_{n}] = \E[X]$.

In the course of this paper, we will often use random variables which are 
uniformly distributed on a given temporal interval $(a,b)$. To denote such a 
random variable $\tau \colon \Omega \to \R$ we write $\tau \sim 
\mathcal{U}(a,b)$.

For a deeper insight of the probabilistic background, we refer the reader to 
\cite{klenke2014}.

The following inequalities will be helpful in order to give suitable a priori 
bounds for the solution of a differential equation and the solution of a 
numerical scheme. 
\begin{lemma}[Discrete Gronwall lemma] \label{lem:discreteGronwall}
	Let $(u_n)_{n\in \N}$ and $(b_n)_{n\in \N}$ be two nonnegative sequences
	which satisfy, for given $a\in [0,\infty)$ and $N \in \N$, that
	\begin{align*}
		u_n \leq a + \sum_{j=1}^{n-1} b_j u_j, 
		\quad \text{for all } n \in \{1,\dots, N \}.
	\end{align*}	
	Then, it follows that
	\begin{align*}
		u_n \leq a \exp \Big( \sum_{j=1}^{n-1} b_j \Big), 
		\quad \text{for all } n \in \{1,\dots, N \},
	\end{align*}
  where we use the convention $\sum_{j=1}^{0} b_j = 0$.
\end{lemma}

\begin{lemma}[Gronwall lemma] \label{lem:Gronwall}
	If $u, a \in C([0,T])$ are nonnegative functions which 
	satisfy, for given  $b \in [0,\infty)$, that
	\begin{align*}
		u(t) \leq a(t) + b\int_{0}^{t}  u(s) \diff{s}, \quad \text{for 
		every } t\in [0,T], 
	\end{align*}
  then
	\begin{align*}
		u(t) \leq \ee^{ b t} \max_{s\in[0,t]}a(s) , \quad \text{for every } 
		t\in [0,T].
	\end{align*}
\end{lemma}
For a proof of the discrete Gronwall lemma, we 
refer the reader to \cite{clark1987}. A proof of Lemma~\ref{lem:Gronwall}
can be found in \cite{hale1980}.


\section{A Carath\'eodory type ODE under a one-sided Lipschitz condition}
\label{sec:ODE}

In this section, we introduce an initial value problem involving an ordinary
differential equation with a non-autonomous vector field of Carath\'eodory
type, that satisfies a one-sided Lipschitz condition. We give a precise
statement of all conditions on the coefficient function in 
Assumption~\ref{as:ODEf}, which are sufficient
to ensure the existence of a unique global solution. The same conditions will
also be used for the error analysis of the randomized backward Euler method in
Section~\ref{sec:randEul}. Further, we briefly investigate the temporal
regularity of the solution $u$. 

Let $T \in (0,\infty)$. We are interested in finding an absolutely continuous
mapping $u \colon [0,T] \to \R^d$ that is a solution to the initial value
problem  
\begin{align}
  \label{eq3:ODE}
  \begin{split}
    \begin{cases}      
      \dot{u}(t) = f(t,u(t)), \quad \text{for almost all } t \in (0,T],\\
      u(0) = u_0,
    \end{cases}
  \end{split}
\end{align}
where $u_0 \in \R^d$ denotes the initial value. The
following conditions on the right-hand side $f \colon [0,T] \times \R^d
\to \R^d$ will ensure the existence of a unique global solution.

\begin{assumption}
  \label{as:ODEf}
  The mapping $f \colon [0,T] \times \R^d \to \R^d$ is measurable. Moreover,
  there exists a null set $\NN_f \in \B([0,T])$ such that:
  \begin{itemize}
    \item[(i)] There exists $\nu \in [0, \infty)$ such that 
      \begin{align*}
			  \big( f(t, x) - f(t,y), x -y \big) \le \nu | x - y |^2,
      \end{align*}
      for all $x, y \in \R^d$ and $t \in [0,T] \setminus \NN_f$.
    \item[(ii)] There exists a mapping $g \colon [0,T] \to
      [0,\infty)$ with $g \in L^2(0,T;\R)$ such that
      \begin{align*}
        | f(t, 0) | \le g(t), \quad \text{ for all } t \in [0,T] \setminus 
        \NN_f.
      \end{align*}
    \item[(iii)] For every compact set $K \subset \R^d$ there exists a mapping
      $L_K \colon [0,T] \to [0,\infty)$ with $L_K \in L^2(0,T;\R)$ such that 
      \begin{align*}
        | f(t, x) - f(t,y) | \le L_K(t) |x -y|
      \end{align*}
      for all $x, y \in K$ and $t \in [0,T] \setminus \NN_f$.
  \end{itemize}  
\end{assumption}

First, we note that from Assumption~\ref{as:ODEf} (i) and
(ii) we immediately get
\begin{align}
  \label{eq3:coerc}
  \begin{split}
    \big( f(t,x) , x \big)
    &= \big( f(t,x) - f(t,0) , x - 0 \big) + \big( f(t,0) , x \big)\\
    &\le \nu |x|^2 +  g(t)  |x|
  \end{split}
\end{align}
for all $x \in \R^d$ and $t \in [0,T] \setminus \NN_f$. 

Moreover, it is well-known that Assumption~\ref{as:ODEf} (ii) and (iii) are
sufficient to ensure the existence of a unique local solution $u \colon [0,T_0)
\to \R^d$ to the initial value 
problem \eqref{eq3:ODE} with a local existence time $T_0 \le T$, see for
instance \cite[Chap.~I, Thm~5.3]{hale1980}.
Here, we recall that a mapping $u \colon [0,T_0) \to \R^d$ is a (local)
\emph{solution in the sense of Carath\'eodory} to \eqref{eq3:ODE} if $u$ is 
absolutely continuous and
satisfies  
\begin{align}
  \label{eq3:sol}
  u(t) = u_0 + \int_0^t f(s,u(s)) \diff{s}
\end{align}
for all $t \in [0,T_0)$. Moreover, for almost all $t \in [0,T_0)$ with $|u(t)|
> 0$ we have
\begin{align*}
  |u(t)| \frac{\diff}{\diff{t}} |u(t)|
  = \frac{1}{2} \frac{\diff}{\diff{t}} | u(t)|^2
  = \big( f(t, u(t)), u(t) \big)
  \le \nu |u(t)|^2 +  g(t)  |u(t)|,
\end{align*}
due to \eqref{eq3:coerc}. Hence, by canceling $|u(t)|>0$ from both sides of
the inequality we obtain 
\begin{align*}
  \frac{\diff}{\diff{t}} |u(t)| \le \nu |u(t)| +  g(t)
\end{align*}
for almost all $t \in [0,T_0)$ with $|u(t)|> 0$. After integrating this 
inequality from $0$ to $t$ it follows
\begin{align*}
	|u(t)| \le |u_0| + \int_{0}^{t}g(s) \diff{s} + \int_{0}^{t} \nu |u(s)| 
	\diff{s} ,
\end{align*}
which holds for all $t\in [0,T_0)$. An application of the Gronwall lemma 
(Lemma~\ref{lem:Gronwall}) yields
\begin{align}
  \label{eq3:growth_u}
  |u(t)| \le \ee^{\nu t} \Big( |u_0| + \int_0^t g(s) \diff{s} \Big)
\end{align}
for all $t \in [0,T_0)$. In particular, since $g \in L^2(0,T;\R)$ we deduce
from \eqref{eq3:growth_u} that $u$ is in fact the unique global solution with 
$T_0 = T$.

Finally, let us investigate the regularity of the solution $u$. To this
end, we define 
\begin{align}
  \label{eq3:Ku}
  K_u := \Big\{ x \in \R^d \, : \, | x | \le  \ee^{\nu T} \Big( |u_0| +
  \int_0^T g(s) \diff{s} \Big) \Big\}.
\end{align}
Clearly, $K_u \subset \R^d$ is a compact set, that contains the origin and 
the complete curve $[0,T] \ni t \mapsto u(t) \in \R^d$ due to
\eqref{eq3:growth_u}. Then, an application of Assumption~\ref{as:ODEf} (iii)
with $K = K_u$ yields  
\begin{align}
  \label{eq3:growth_f}
  | f(t,u(t)) | \le L_{K_u}(t) | u(t) | + |f(t, 0) |
  \le L_{K_u}(t) | u(t) | + g(t)
\end{align}
for all $t \in [0,T] \setminus \NN_f$. 

For arbitrary $s, t \in [0,T]$ with $s < t$ it follows from
\eqref{eq3:sol} that 
\begin{align*}
  | u(s) - u(t) | \le \int_s^t | f( z , u(z) ) |  \diff{z}.
\end{align*}
Furthermore, after inserting \eqref{eq3:growth_f}, we have
\begin{align*}
  | u(s) - u(t) | &\le \int_s^t L_{K_u}(z) | u(z) | +  g(z)  \diff{z}\\
  &\le \Big( 1 + \sup_{z \in [0,T]} |u(z)| \Big) \int_s^t \big(L_{K_u}(z) +
  g(z) \big) \diff{z}.
\end{align*}
Then, an application of the Cauchy--Schwarz
inequality yields
\begin{align}
  \label{eq3:hoelder_u}
  | u(s) - u(t) | &\le \big( 1 + \|u\|_{C([0,T];\R^d)} \big) \big\|L_{K_u}
  + g \big\|_{L^2(0,T;\R)} | s - t |^{\frac{1}{2}}
\end{align}
for all $s, t\in [0,T]$. This proves that $u$ is H\"older continuous with
exponent $\frac{1}{2}$.

\section{Error analysis of the randomized backward Euler method}
\label{sec:randEul}

This section is devoted to the error analysis of the randomized backward Euler
method. Our error analysis partly relies on variational methods developed in 
\cite{emmrich2009}, that have recently been adapted to stochastic problems in 
\cite{andersson2017}. 

In this section, we consider the following randomized version of the backward
Euler method: Let $N \in \N$ denote the number of temporal steps and set $k =
\frac{T}{N}$ as the temporal step size. For given $N$ and $k$ we obtain
an equidistant partition of the interval $[0,T]$ given by $t_n := k n$, $n\in 
\{0,\dots, N \}$.
Further, let $\tau = (\tau_{n})_{n \in \N}$ be
a family of independent and $\mathcal{U}(0,1)$-distributed random
variables on a complete probability space $(\Omega,\F,\P)$ and let $\tn =
(\tn_{n})_{n  \in \N}$ be the family of random variables given by $\tn_n = t_n
+  k\tau_n$ for $n\in\N$. Then the numerical approximation $(U^n)_{n \in 
\{0,\dots, N \} }$ of
the solution $u$ is determined by the recursion 
\begin{align}
  \label{eq4:RandBackEuler}
  \begin{split}
	  \begin{cases}
	    U^n = U^{n-1} + k f( \tn_n, U^n ),\quad \text{ for } 
	    n \in \{ 1,\ldots,N \} ,\\ 
	    U^0 = u_0.
	  \end{cases}
  \end{split}
\end{align}
When investigating the solvability of this implicit equation, the mild step 
size restriction $k \nu < 1$ becomes necessary due to the implicit structure 
of the scheme. When considering a dissipative equation which is the case 
when $\nu \leq 0$ the restriction disappears. This case corresponds to the 
setting of the monotone operators in Section~\ref{sec:nonlinPDE}.

Note that \eqref{eq4:RandBackEuler} is an implicit Runge--Kutta
method with one stage and a randomized node. More precisely, in each step 
we apply one member of the following family of implicit Runge--Kutta
methods determined by the Butcher tableau
\begin{align}
  \label{eq4:randRK}
  \begin{split}
    \begin{array}{c|c}
      \theta & 1 \\
      \hline
      & 1 
    \end{array}
  \end{split}
\end{align}
where the value of the parameter $\theta \in [0,1]$ is determined by the random
variable $\tau_j$ in the $j$-th step.

Further, the resulting sequence $(U^n)_{n \in
\{0,\ldots,N\}}$ consists of random variables, since we artificially
inserted randomness into the numerical method. From a probabilistic point of
view, $(U^n)_{n \in \{0,\ldots,N\}}$ is in fact a discrete time stochastic
process, that takes values in $\R^d$ and is adapted to the complete 
filtration $(\F_n)_{n\in \N}$. Here, $\F_n \subset \F$ is 
the smallest complete $\sigma$-algebra such that the subfamily $(\tau_j)_{j \in
\{1,\ldots,n\}}$ is measurable. Note that $\F_n  \subset \F_m$, whenever $n \le
m$. More precisely,
\begin{align}
  \label{eq3:filtration}
  \begin{split}
    \F_0 &:= \sigma \big( \NN \in \F\; : \; \P(\NN) = 0 \big), \\
    \F_n &:= \sigma \big( \sigma( \tau_j\; : \; j \in \{ 1,\ldots , n \}) \cup
    \F_{0}\big), \quad n\in \N.
  \end{split}
\end{align}
In particular, each $\P$-null set (and each subset of a $\P$-null set) is
contained in every $\sigma$-algebra $\F_n$, $n \in \N_0$.

Next, let us introduce the following set $\mathcal{G}^2_N$ of
\emph{square-integrable and adapted grid functions}. For each $N \in \N$ this
set is defined by
\begin{align*}
  \mathcal{G}^2_N := \big\{ &Z \colon \{0,\ldots,N\} \times \Omega \to \R^d \, 
  :
  \, Z^0 = z_0 \in \R^d, \\
  &Z^n, f(\tn_n, Z^n) \in L^2(\Omega,\F_{n}, \P;\R^d) \text{
  for } n \in \{1,\ldots,N\} \big\}.
\end{align*}
Take note that $z_0 \in \R^d$ is an arbitrary deterministic initial 
value and that the condition $Z^n \in L^2(\Omega,\F_{n}, \P;\R^d)$ ensures that
$Z^n$ is square-integrable as well as measurable with respect to the 
$\sigma$-algebra $\F_n$. First, we will show that 
the randomized backward Euler method \eqref{eq4:RandBackEuler}
with a sufficiently large number $N \in \N$ of steps uniquely
determines an element in $\mathcal{G}^2_N$.  


We begin by proving the existence of a solution to the implicit scheme. First,
we state two technical lemmata to prove the existence and measurability of 
a 
solution.

\begin{lemma}
  \label{lem:root}
  For $R \in (0,\infty)$ let $h\colon \overline{B_R(0)} \subseteq \R^d \to 
  \R^d$ be 
  continuous 
  and fulfill the condition
  \begin{align*}
    \ska{h(x)}{x} \geq 0, \quad \text{ for every } x \in \partial B_R(0).
  \end{align*}
  Then there exists at least one $x_0\in \overline{B_R(0)}$ such that 
  $h(x_0)=0$.
\end{lemma}

A proof of Lemma~\ref{lem:root} is found, for instance, in 
\cite[Sec.~9.1]{evans1998}. 

\begin{remark} \label{rem:InnerProd}
	For a symmetric, positive definite $Q \in \R^{d,d}$ Lemma~\ref{lem:root} 
	can be extended as follows. If a function $h\colon B_{Q,R} \subseteq \R^d 
	\to \R^d$ where $B_{Q,R}$ is given by
	\begin{align*}
		B_{Q,R} = \{ x\in \R^d \, : \,\ska{Qx}{x} \leq R^2 \}
	\end{align*}
	is continuous and fulfills
	\begin{align*}
		\ska{Q h (x)}{x} \geq 0, \quad \text{ for every } x\in \partial B_{Q,R},
	\end{align*}
	then there exists $x_0 \in B_{Q, R}$ such that $h(x_0) = 0 $.
	This extension of Lemma~\ref{lem:root} can be proved by exploiting that
	\begin{align*}
		\ska{Q h (x)}{x} \geq 0, \quad \text{ for every } x\in \R^d \text{ with } 
		\ska{Qx}{x} = R^2,
	\end{align*}
	can be rewritten as
	\begin{align*}
		\ska{Q^{\frac{1}{2}} h (Q^{-\frac{1}{2} } y)}{y} \geq 0, \quad \text{ for 
		every } y \in \R^d \text{ with } \ska{y}{y} = R^2,
	\end{align*}
	using the transformation $y = Q^{\frac{1}{2}}x$.
\end{remark}

The next result is needed in order to prove the measurability of the sequence
generated by the implicit numerical method \eqref{eq4:RandBackEuler}. For a
closely related result we refer to \cite[Lem.~3.8]{gyoengy1982}. The proof
presented here follows an approach from \cite[Prop.~1]{DindosToma1997},
that can easily be extended to more general situations. 

\begin{lemma}\label{lem:measurable}
	Let $\tilde{\mathcal{F}}$ be a complete sub $\sigma$-algebra of the 
	$\sigma$-algebra
	$\F$, $\mathcal{M} \in \tilde{\mathcal{F}}$ with $\P(\mathcal{M}) = 1$ and $h
	\colon \Omega \times \R^d \to \R^d$ such that the following conditions are
	fulfilled. 
	\begin{enumerate}[label=(\roman*)]
		\item \label{item:cont}
		The mapping $x \mapsto h(\omega, x)$ is continuous for every $\omega 
		\in \mathcal{M}$. 
		\item \label{item:measurable} 
		The mapping $\omega \mapsto h(\omega, x)$ is 
		$\tilde{\mathcal{F}}$-measurable for every $x\in \R^d$.
		\item \label{item:root} 
		For every $\omega \in \mathcal{M}$ there exists a unique root of the 
		function $h(\omega,\cdot)$.
	\end{enumerate}	
  Define the mapping 
  \begin{align*}
    U \colon \Omega \to\R^d, \quad \omega \mapsto U(\omega),
  \end{align*}
  where $U(\omega)$ is the unique root of $h(\omega,\cdot)$ for $\omega \in 
  \mathcal{M}$ and $U(\omega)$ is arbitrary for $\omega \in \Omega\setminus 
  \mathcal{M}$. 
  Then $U$ is $\tilde{\mathcal{F}}$-measurable.
\end{lemma}
  
\begin{proof} 
  Define the (multivalued) mapping  
  \begin{align*}
    U_{\varepsilon}\colon \Omega \to \Pow(\R^d), \quad 
    U_{\varepsilon}(\omega) &:= \{ x\in \R^d\, : \,h(\omega,x) \in   
    B_{\varepsilon}(0) \}
  \end{align*}
  for $\varepsilon >0$. We first show for an arbitrary open set $A \in 
  \B(\R^d)$ that the set
  \begin{align*}
    U_{\varepsilon}^{-1}(A) &:= 
    \{ \omega \in \Omega\, : \,\text{ there exists } 
    x\in A \text{ such that } h(\omega, x) \in B_{\varepsilon}(0) \}\\
    &= \bigcup_{x \in A} \{ \omega \in \Omega \, : \,h(\omega,x) \in 
    B_{\varepsilon}(0) \}
  \end{align*}
	is an element of $\tilde{\mathcal{F}}$. To this end, first note
	that $h(\cdot, x)^{-1}(B_{\varepsilon}(0)) \in \tilde{\mathcal{F}}$
	since $\omega \mapsto h(\omega,x)$ is $\tilde{\mathcal{F}}$-measurable.
  Then, it follows that
  \begin{align*}
    U_{\varepsilon}^{-1}(A \cap \Q^d) 
    &= \bigcup_{x \in A\cap \Q^d} \{ \omega \in \Omega \, : \,h(\omega,x) \in 
    B_{\varepsilon}(0) \} \\
    &= \bigcup_{x \in A\cap \Q^d} h(\cdot, x)^{-1}(B_{\varepsilon}(0)) 
		\in \tilde{\mathcal{F}}.
  \end{align*}
  It remains to verify the equality
  \begin{align}
    \label{eq4:seteq}
    U_{\varepsilon}^{-1}(A) = U_{\varepsilon}^{-1}(A \cap \Q^d).
  \end{align}
  It is clear that $U_{\varepsilon}^{-1}(A \cap \Q^d)$ is a subset of 
  $U_{\varepsilon}^{-1}(A)$. 
  
  To prove $U_{\varepsilon}^{-1}(A) \subseteq 
  U_{\varepsilon}^{-1}(A  \cap \Q^d)$ we consider two cases. If 
  $U_{\varepsilon}^{-1}(A)$ is a subset of $\Omega \setminus \mathcal M$ 
	then it is a null set and lies in $\tilde{\mathcal{F}}$
	due to the completeness of the $\sigma$-algebra. 
  Else, we can assume that there exist $\omega \in 
  U_{\varepsilon}^{-1}(A) \cap \mathcal M$ 
  and $x_0 \in A$ with $h(\omega,x_0) \in B_{\varepsilon}(0)$.
  In particular, we note that the function $x\mapsto h(\omega,x)$ is
  continuous, since $\omega \in \mathcal{M}$.
  Further, observe that $A$ is an open neighborhood of $x_0$ and
  $B_{\varepsilon}(0)$ is an open neighborhood of $h(\omega,x_0)$.
  Since $B_{\varepsilon}(0)$ is open, the continuity of $h$ 
  implies that the set 
  \begin{align*}
    C := h(\omega, \cdot )^{-1}(B_{\varepsilon}(0))
  \end{align*}
  is an open set in $\R^d$ with $x_0 \in C$. Thus, $C \cap A$ is nonempty and
  open. Therefore, there exists
  $\overline{x} \in ( C \cap A ) \cap \Q^d$ such that 
  $h(\omega,\overline{x}) \in B_{\varepsilon}(0)$. This
  implies $ \omega \in U_{\varepsilon}^{-1}(A  \cap \Q^d)$ and completes
  the proof of \eqref{eq4:seteq}. Consequently, $U_{\varepsilon}^{-1}(A)
	\in \tilde{\mathcal{F}}$ for each open set $A \in \B(\R^d)$.
  
  Next, recall that for each $\omega \in \mathcal M$ the image of $U$ is
  defined as the unique element of $h(\omega,\cdot)^{-1} (\{ 0 \})$. Thus,
  the set 
  \begin{align*}
    U_0(\omega) := \bigcap_{j \in \N} U_{\frac{1}{j}}(\omega)
  \end{align*}
  consists of a single element which coincides with $U(\omega)$. 
  Therefore we obtain
  \begin{align*}
    \mathcal{M} \cap U^{-1}(A)
    &= \mathcal{M} \cap \{ \omega \in \Omega :\text{ there exists } x\in A 
    \text{ such that } h(\omega,x) = 0
    \} \\
    &= \mathcal{M} \cap \bigcap_{j \in \N} \{ \omega \in \Omega \, : \,\text{ 
    there 
    exists } x\in A \text{ such that }  h(\omega,x) \in  B_{\frac{1}{j}}(0)\} 
    \\
    &= \mathcal{M} \cap \bigcap_{j \in \N} U_{\frac{1}{j}}^{-1}(A)
    = \mathcal{M} \cap \bigcap_{j \in \N} U_{\frac{1}{j}}^{-1}(A\cap 
    \Q^d) \in \tilde{\mathcal{F}},
  \end{align*}
	which also implies $U^{-1}(A) \in \tilde{\mathcal{F}}$
	for every open set $A \in \B(\R^d)$ due to the completeness 
	of $\tilde{\mathcal{F}}$. From this the measurability of
  the mapping $\omega \mapsto U(\omega)$ follows.
\end{proof}  

\begin{lemma}
  \label{lem:exist}
  Let Assumption~\ref{as:ODEf} be satisfied. Then for each $N\in\N$ with $ 
  \frac{T}{N} \nu =k \nu < 1$ there  exists a unique solution $U= 
  (U^n)_{n\in \{0,\dots, N \}} \in \mathcal{G}^2_N$ to 
	the implicit scheme \eqref{eq4:RandBackEuler}.
\end{lemma}

\begin{proof}
  The assertion $U^n \in L^2(\Omega,\F_n,\P ;\R^d)$ is proved using an 
  inductive argument for $n \in \{ 0,\dots,N \}$.  Since $U^0 \equiv u_0 \in 
  L^2(\Omega,\F_0,\P ;\R^d)$ the
	case $n = 0$ is evident. Next, assuming $U^{n-1} \in L^2(\Omega,\F_{n-1},\P
	;\R^d)$ exists, we define the set
  \begin{align*}
    \mathcal{M} = \{ \omega \in \Omega\, : \,g( \tn_n(\omega) ) < \infty, 
    \ |U^{n-1}(\omega)| < \infty \text{ and }
    \tn_n(\omega) \in [0,T] \setminus \NN_f \},
  \end{align*}
  where $\NN_f \in \B([0,T])$ is the null set from Assumption~\ref{as:ODEf}.
  Since $\|g\|_{L^2(0,T;\R)} < \infty$ and $\|U^{n-1}\|_{L^2(\Omega;\R^d)} < 
  \infty$ the set fulfills $\P (\mathcal{M}) =1$. We define the function $h_n$ 
  by
  \begin{equation}
    \label{eq:defH}
    h_n \colon \Omega \times \R^d \to \R^d, \quad h_n(\omega, x) = x - 
    U^{n-1}(\omega) 
    - kf(\tn_n(\omega), x).
  \end{equation}
  In the following we consider a fixed $\omega \in \mathcal{M}$. Then the 
  mapping $h_n(\omega,\cdot)$ is continuous by Assumption~\ref{as:ODEf} (iii). 
  Further we write 
  \begin{align*}
    R = R(\omega) = \frac{1}{1-\nu k} (| U^{n-1}(\omega) | + k g( 
    \tn_n(\omega) )).
  \end{align*}
  Thus, for each $ x\in \R^d$ with $|x| = R$ this implies
  \begin{align*}
    \ska{h_n(\omega, x)}{x} 
    & = |x|^2 - \ska{U^{n-1}(\omega) }{x} - k \ska{f( \tn_n(\omega ), x )}{x}\\
    & \geq R^2 -  |U^{n-1}(\omega)| R - k \nu R^2 - k g(\tn_n(\omega)) R\\  
    & = R^2 - k \nu R^2  - (|U^{n-1}(\omega)| + k g(\tn_n(\omega))) R \\ 
    & \geq (1-\nu k) R^2 - (1-\nu k) R^2 = 0. 
  \end{align*}
  Hence, by Lemma~\ref{lem:root}, for every $\omega \in \mathcal{M}$ there 
  exists $x = x(\omega) \in 
  \R^d$ such that $h_n(\omega, x)=0$ holds. This $x$ is always unique: Assume
  there exists $\omega \in \mathcal{M}$ and $x,y \in  \R^d$ such that
  \begin{align*}
    x = U^{n-1}(\omega) + kf(\tn_n(\omega), x) 
    \quad \text{ and } \quad
    y = U^{n-1}(\omega) + kf(\tn_n(\omega), y) 
  \end{align*}
  hold. Then we can write for the difference
  \begin{align*}
    |x -y|^2 
    & = k \big( f(\tn_n(\omega), x) - f(\tn_n(\omega), y), x - y \big)\\
    & \leq k \nu |x - y|^2 < |x - y|^2 
  \end{align*}
  which implies $x = y$. Thus, the function $h_n \colon \Omega \times \R^d \to 
  \R^d$ is $\F_n$-measurable in the first entry, continuous in the second and 
  has a unique root $x$ for every $\omega \in \mathcal{M}$. 
  Then, Lemma~\ref{lem:measurable} implies that the function 
  \begin{align*}
    U^n \colon \Omega \to \R^d, \quad \omega \mapsto U^n(\omega),
  \end{align*}
  where $U^n(\omega)$ is the unique root of $h_n(\omega, \cdot)$ for $\omega 
  \in \mathcal{M}$ and $U^{n-1}(\omega)$ for $\omega \in \Omega \setminus 
  \mathcal{M}$ is $\F_n$-measurable.  

  It remains to prove that $U^n$ is finite with respect to the
  $L^2(\Omega;\R^d)$-norm. Using \eqref{eq3:coerc}, it follows
  \begin{align*}
    \| U^n \|_{L^2(\Omega;\R^d)}^2
    & = \E [ \ska{U^{n-1}}{U^n} + k \ska{ f(\tn_n, U^n) }{ U^n }] \\
    & \leq \|U^{n-1}\|_{L^2(\Omega;\R^d)} \| U^n \|_{L^2(\Omega;\R^d)} 
    + k  \E [ \nu | U^n|^2 + g(\tn_n) |U^n|] \\
    & \leq \|U^{n-1}\|_{L^2(\Omega;\R^d)} \| U^n \|_{L^2(\Omega;\R^d)} 
    + k \nu \| U^n\|^2_{L^2(\Omega;\R^d)} \\
    & \quad  + \sqrt{k} \|g\|_{L^2(0,T;\R)} \|U^n 
    \|_{L^2(\Omega;\R^d)}
  \end{align*}
  and therefore
  \begin{align*}
    \|U^n\|_{L^2(\Omega;\R^d)}
    & \leq \frac{1}{1 - k \nu } \big(\|U^{n-1}\|_{L^2(\Omega;\R^d)} + \sqrt{k} 
    \|g\|_{L^2(0,T;\R)} \big).
  \end{align*} 
  The last step is to prove that the function $f(\tn_n, U^n)$ also lies in
  $L^2(\Omega, \F_n,\P;\R^d)$. The mapping $\omega \mapsto f(\tn_n(\omega),
  U^n(\omega))$ is $\F_n$-measurable since $f$ is measurable and both 
  $\tn_n$ and $U^n$ are $\F_n$-measurable. 
  Since both $U^n$ and $U^{n-1}$ are elements of $L^2(\Omega;\R^d)$ we can 
  write
  \begin{align*}
    \|f(\tn_n,U^n) \|_{L^2(\Omega;\R^d)} 
    = \| \tfrac{1}{k}(U^n - U^{n-1}) \|_{L^2(\Omega;\R^d)}.
  \end{align*}
  Thus, $f(\tn_n,U^n)$ is finite in the $L^2(\Omega;\R^d)$-norm.
\end{proof}


The following stability lemma will play an important role in the error
analysis of the randomized backward Euler method. Its proof is based on
techniques developed in \cite{andersson2017}. For its formulation we
introduce the \emph{local residual} $(\rho_N^n(V))_{n \in \{0,\ldots,N\}}$, $N
\in \N$, of an arbitrary grid function $V = (V^n)_{n \in \{0,\ldots,N\}} \in
\mathcal{G}^2_N$. More precisely, for every $n \in \{1,\ldots,N\}$ we define
$\rho_N^n(V)$ by  
\begin{align}
  \label{eq:locresidual}
  \rho_N^n(V) = k f(\tn_n, V^n) - V^n + V^{n-1}.
\end{align}
Since $(V^n)_{n \in \{0,\ldots,N\}} \in \mathcal{G}^2_N$ it directly follows 
that $\rho_N^n(V) \in L^2(\Omega,\F_n,\P;\R^d)$ for every $n \in
\{0,\ldots,N\}$.

\begin{lemma}
  \label{lem:stab}
  Let Assumption~\ref{as:ODEf} be satisfied. For $N \in \N$ let $(U^n)_{n \in
  \{0,\ldots,N\}} \in \mathcal{G}^2_N$ be the grid function generated by
  \eqref{eq4:RandBackEuler} with step size $k = \frac{T}{N}$. If $\nu k <
  \frac{1}{4}$, then for every $V \in \mathcal{G}^2_N$ it holds true that
  \begin{align*}
    \big\| U^n - V^n \big\|_{L^2(\Omega;\R^d)}&\\
    \le \ee^{(2 \nu + 1) t_n } \bigg(&  \big| U^0 - V^0 \big|^2 \\
    & +  \sum_{j = 1}^n 
    \Big( 2 \big\| \rho_N^j(V) \big\|^2_{L^2(\Omega;\R^d)}  
    + \frac{2}{k} \big\| \E [ \rho_N^j(V) | \F_{j-1} 
    ]\big\|^2_{L^2(\Omega;\R^d)} \Big)
    \bigg)^{\frac{1}{2}}
  \end{align*}
  for every $n \in \{1, \ldots, N\}$.
\end{lemma}

\begin{proof}
  Let $N \in \N$ and $V = (V^j)_{j \in \{0,\ldots,N\}} \in \mathcal{G}^2_N$ be
  arbitrary. Set $E^j := U^j - V^j$ for each $j \in \{0,\ldots,N\}$.
  Since $(a -b, a) = \frac{1}{2} \big(  | a |^2 - | b |^2 +  | a - b
  |^2 \big)$ for all $a, b \in \R^d$ we get for every $j \in \{1,\ldots,N\}$
  \begin{align*}
    &|E^j|^2 - | E^{j-1} |^2 + | E^j - E^{j-1}|^2
    = 2 \big( E^j - E^{j-1}, E^j \big) \\
    &\quad = 2 k \big( f(\tn_j, U^{j} ) 
    - f(\tn_j, V^j ) , E^j \big)  + 2 \big( k f( \tn_j, V^j ) 
    - V^j + V^{j - 1}, E^j \big).
  \end{align*}
  Next, note that $\P( \tn_j \in \NN_f ) = 0$, where
  $\NN_f$ denotes the null set from Assumption~\ref{as:ODEf}.
  Hence, we can apply Assumption~\ref{as:ODEf} (i) to the first
  term on a set with probability one. In addition, we insert
  \eqref{eq:locresidual} into the second term and 
  obtain the inequality
  \begin{align*}
    &|E^j|^2 - | E^{j-1} |^2  +  | E^j - E^{j-1}|^2\\ 
    &\quad \le 2 \nu k | E^j |^2 +  2 \big( \rho_N^j(V), E^j - E^{j-1} \big)
    + 2 \big( \rho_N^j(V), E^{j-1} \big) \quad \text{almost surely.}
  \end{align*}
  After taking the expected value we further observe that
  \begin{align*}
    \E \big[ \big( \rho_N^j(V), E^{j-1} \big) \big]
    &= \big\langle \rho_N^j(V), E^{j-1} \big\rangle_{L^2(\Omega;\R^d)} 
    =  \big\langle \E [ \rho_N^j(V) | \F_{j-1} ], E^{j-1} 
    \big\rangle_{L^2(\Omega;\R^d)},
  \end{align*}
  since $E^{j-1}$ is $\F_{j-1}$-measurable. Then, applications of the
  Cauchy--Schwarz inequality and the weighted Young inequality yield
  \begin{align*}
    &2 \big\langle \E [ \rho_N^j(V) | \F_{j-1} ], E^{j-1} 
    \big\rangle_{L^2(\Omega;\R^d)}\\
    &\quad \le \frac{1}{k} \big\| \E [ \rho_N^j(V) | \F_{j-1} ] 
    \big\|^2_{L^2(\Omega;\R^d)}
    + k \big\| E^{j-1} \big\|^2_{L^2(\Omega;\R^d)}.
  \end{align*}
  In the same way, the Cauchy--Schwarz and Young inequalities also yield
  \begin{align*}
    \E \big[ 2 \big( \rho_N^j(V), E^j - E^{j-1} \big) \big]
    \le \big\| \rho_N^j(V) \big\|^2_{L^2(\Omega;\R^d)} 
    + \big\| E^j - E^{j-1} \big\|^2_{L^2(\Omega;\R^d)}.    
  \end{align*}
  Altogether, we have shown that
  \begin{align*}
    &\|E^j\|^2_{L^2(\Omega;\R^d)} - \| E^{j-1} \|^2_{L^2(\Omega;\R^d)}\\ 
    &\quad \le  2 \nu k \| E^j \|^2_{L^2(\Omega;\R^d)}
    + \big\| \rho_N^j(V) \big\|^2_{L^2(\Omega;\R^d)} \\
    & \qquad +  \frac{1}{k} \big\| \E [ \rho_N^j(V) | \F_{j-1} ] 
    \big\|^2_{L^2(\Omega;\R^d)}
    + k \big\| E^{j-1} \big\|^2_{L^2(\Omega;\R^d)},
  \end{align*}
  where we canceled the term $\| E^j - E^{j-1} \|^2_{L^2(\Omega;\R^d)}$ on both
  sides of the inequality. Then, after some rearrangements and summing this
  inequality for all $j \in  \{1,\ldots,n \}$ with arbitrary $n \in  \{1,
  \ldots, N\}$ we obtain 
  \begin{align*}
    (1 - 2 \nu k ) \| E^n \|^2_{L^2(\Omega;\R^d)} 
    &\le (1 - 2 \nu k ) \| E^0 \|^2_{L^2(\Omega;\R^d)} 
    + (2 \nu + 1) k \sum_{j = 1}^{n} \big\|  E^{j - 1} 
    \big\|^2_{L^2(\Omega;\R^d)}\\
    &\quad + \sum_{j = 1}^n \Big( \big\| 
    \rho_N^j(V)\big\|^2_{L^2(\Omega;\R^d)}  
    + \frac{1}{k} \big\| \E [ \rho_N^j(V) | \F_{j-1} ] 
    \big\|^2_{L^2(\Omega;\R^d)} \Big). 
  \end{align*}
  Next, note that from the assumption $\nu k <
  \frac{1}{4}$ it follows $(1 - 2 \nu k)^{-1} \le 2$. Therefore,
  \begin{align*}
    \| E^n \|^2_{L^2(\Omega;\R^d)} 
    &\le \| E^0 \|^2_{L^2(\Omega;\R^d)} 
    + 2 (2 \nu + 1) k \sum_{j = 1}^{n} 
    \big\|  E^{j - 1} \big\|^2_{L^2(\Omega;\R^d)}\\
    &\quad +  \sum_{j = 1}^n \Big( 2 
    \big\| \rho_N^j(V) \big\|^2_{L^2(\Omega;\R^d)}  
    + \frac{2}{k} \big\| \E [ \rho_N^j(V) | \F_{j-1} ] 
    \big\|^2_{L^2(\Omega;\R^d)} \Big).
  \end{align*}
  Finally, applying a discrete Gronwall lemma 
  (Lemma~\ref{lem:discreteGronwall}) completes the proof.
\end{proof}


The second ingredient in the error analysis is an estimate of the local
residual of the exact solution. For its formulation we need to represent the
exact solution by a grid function. This is easily achieved by restricting $u$
to the temporal points $t_n = n k$, $n \in \{0,\ldots,N\}$, with $k =
\frac{T}{N}$ and $N \in \N$. More precisely, we define the \emph{restriction}
$u|_N$ of $u$ to the grid points $(t_n)_{n \in \{0,\ldots,N\}}$ by  
\begin{align}
  \label{eq4:restriction}
  [u|_N]^n := u(t_n) 
\end{align}
for all $n \in \{0,\ldots,N\}$. Since $u|_N$ is deterministic
we clearly have $[u|_N]^n = u(t_n) \in
L^2(\Omega,\F_n, \P;\R^d)$ for all $n \in \{0,\ldots,N\}$.
In addition, as in \eqref{eq3:growth_f} we have
\begin{align*}
  \| f( \tn_n, u(t_n) ) \|_{L^2(\Omega;\R^d)}
  &\le \big\| L_{K_u}(\tn_n) | u(t_{n}) |  + g( \tn_n ) 
  \big\|_{L^2(\Omega;\R)} \\
  & \le \frac{1}{\sqrt{k}} \big( 1 + \|u \|_{C([0,T];\R^d)} \big)  \big\| 
  L_{K_u} + g
  \big\|_{L^2(t_{n-1},t_n;\R)} < \infty.
\end{align*} 
This shows that $u|_N \in \mathcal{G}^2_N$ for every $N \in \N$.

\begin{lemma}
  \label{lem:cons}
  Let Assumption~\ref{as:ODEf} be satisfied. Then, for all $N \in \N$ and 
  $n \in \{1,\ldots, N\}$ the local residual \eqref{eq:locresidual} of the
  exact solution $u$ to the initial value problem \eqref{eq3:ODE} is bounded
  by 
  \begin{align}
    \label{eq4:cons1}
    \begin{split}
      \big\| \rho_N^n(u|_N) \big\|_{L^2(\Omega;\R^d)} 
      &\le \big( 1 + \|u \|_{C([0,T];\R^d)} \big) 
      \Big(1 + T^{\frac{1}{2}} \big\| L_{K_u} + g \big\|_{L^2(0,T;\R)} \Big)
      \\
      &\quad \times \Big( \big\| g \big\|_{L^2(t_{n-1},t_n;\R)} + \big\|
      L_{K_u} \big\|_{L^2(t_{n-1},t_n;\R)} \Big) k^{\frac{1}{2}}
    \end{split}
  \end{align}
  and 
  \begin{align}
    \label{eq4:cons2}
    \begin{split}
    &\big\| \E \big[ \rho_N^n(u|_N)| \F_{n-1} \big] \big\|_{L^2(\Omega;\R^d)}\\
    &\quad \le \big( 1 + \| u \|_{C([0,T];\R^d)} \big)
    \big\| L_{K_u} + g \big\|_{L^2(0,T;\R)} \big\| L_{K_u}
    \big\|_{L^2(t_{n-1},t_n;\R)} k.
    \end{split}
  \end{align}
\end{lemma}

\begin{proof}
  Fix $N \in \N$ and $n \in \{1,\ldots,N\}$ arbitrarily. First recall that
  \begin{align*}
    \rho_N^n(u|_N) = k f(\tn_n, u(t_n) ) - u(t_n) + u(t_{n-1}).
  \end{align*}
  Inserting \eqref{eq3:sol} yields
  \begin{align}
    \label{eq4:rep_res}
    \begin{split}
      \rho_N^n(u|_N) &= k \big( 
      f(\tn_n, u(t_n) ) -  f(\tn_n, u(\tn_n) ) \big) + k f(\tn_n, u(\tn_n) ) - 
      \int_{t_{n-1}}^{t_n} f(s,u(s)) \diff{s}.
    \end{split}
  \end{align}
  Since $f$ and $u$ are deterministic, the only source of randomness in this
  expression is the random variable $\tn_n$. Further, since $\tn_n$ is
  independent of $\F_{n-1}$ we obtain
  \begin{align*}
    \E \big[ \rho_N^n(u|_N) | \F_{n -1} \big] &= \E \big[ \rho_N^n(u|_N) 
    \big]
    = k \E \big[ f(\tn_n, u(t_n) ) -  f(\tn_n,
    u(\tn_n) ) \big],
  \end{align*}
  where we also used that
  \begin{align}
    \label{eq4:mean}
     k \E \big[ f(\tn_n, u(\tn_n) ) \big] = \int_{t_{n-1}}^{t_n} f(s,u(s)) 
     \diff{s}.
  \end{align}
  Since $\P( \tn_n \in \NN_f)= 0$ we can apply 
  Assumption~\ref{as:ODEf} (iii) with the compact
  set $K = K_u \subset \R^d$ defined in \eqref{eq3:Ku} inside the expectation.
  This yields
  \begin{align*}
    \big\| \E \big[ \rho_N^n(u|_N) | \F_{n -1} \big] 
    \big\|_{L^2(\Omega;\R^d)}
    &= \big| \E \big[ \rho_N^n(u|_N) \big] \big| \\
    &\le k \E \big[ | f(\tn_n, u(t_n) ) -  f(\tn_n,
    u(\tn_n) ) | \big] \\
    &\le k \E \big[ L_{K_u}(\tn_n) \, | u(t_n) - u(\tn_n) | \big]\\
    &\le k \big( \E \big[ L_{K_u}(\tn_n)^2 \big]
    \big)^{\frac{1}{2}} \big\| u(t_n) - u(\tn_n) \big\|_{L^2(\Omega;\R^d)}.
  \end{align*}
  Then, we make use of the H\"older continuity \eqref{eq3:hoelder_u} 
  of $u$ and obtain
  \begin{align*}
    \big\| u(t_n) - u(\tn_n) \big\|_{L^2(\Omega;\R^d)} 
    \le  \big( 1 + \| u \|_{C([0,T];\R^d)} \big) \big\| L_{K_u} 
    + g \big\|_{L^2(0,T;\R)} k^{\frac{1}{2}} .
  \end{align*}
  In addition, we note that
  \begin{align*}
    k^{\frac{1}{2}} \big( \E \big[ L_{K_u}(\tn_n)^2 \big]
    \big)^{\frac{1}{2}} 
    = \Big( \int_{t_{n-1}}^{t_n} L_{K_u}(s)^2
    \diff{s} \Big)^\frac{1}{2}.
  \end{align*}
  Hence,
  \begin{align*}
    &\big\| \E \big[ \rho_N^n(u|_N) | \F_{n -1} \big] 
    \big\|_{L^2(\Omega;\R^d)}\\
    &\quad \le \big( 1 + \| u \|_{C([0,T];\R^d)} \big)
    \big\| L_{K_u} + g \big\|_{L^2(0,T;\R)}
    \Big( \int_{t_{n-1}}^{t_n} L_{K_u}(s)^2 \diff{s} \Big)^{\frac{1}{2}} k,
  \end{align*}
  which proves assertion \eqref{eq4:cons2}.

  It remains to show \eqref{eq4:cons1}. To this end, we directly apply the
  $L^2( \Omega;\R^d)$-norm to \eqref{eq4:rep_res} and obtain
  \begin{align}
    \label{eq4:term1}
    \begin{split}
      \big\| \rho_N^n(u|_N) \big\|_{L^2(\Omega;\R^d)}
      &\le k \big\| f(\tn_n, u(t_n) ) -  f(\tn_n, u(\tn_n) ) 
      \big\|_{L^2(\Omega;\R^d)}\\
      &\qquad + \Big\| k f(\tn_n, u(\tn_n) ) - \int_{t_{n-1}}^{t_n} f(s,u(s)) 
      \diff{s} 
      \Big\|_{L^2(\Omega;\R^d)}.
    \end{split}
  \end{align}
  By similar arguments as above we derive the following estimate
  for the first term:
  \begin{align*}
    &k \big\| f(\tn_n, u(t_n) ) -  f(\tn_n,
    u(\tn_n) ) \big\|_{L^2(\Omega;\R^d)}\\
    &\quad \le k \big( \E \big[ L_{K_u}(\tn_n)^2 
    | u(t_n) - u(\tn_n) |^2 \big] \big)^{\frac{1}{2}}\\
    &\quad \le k^{\frac{1}{2}} T^{\frac{1}{2}}
    \big( 1 + \| u \|_{C([0,T];\R^d)} \big)
    \big\| L_{K_u} + g \big\|_{L^2(0,T;\R)}
    \Big( \int_{t_{n-1}}^{t_n} L_{K_u}(s)^2 \diff{s} \Big)^{\frac{1}{2}},
  \end{align*}
  where we also made use of the estimate $k \le T$ in the last step.

  Regarding the second summand in \eqref{eq4:term1} we first observe that
  \begin{align*}
    &\Big\| k f(\tn_n, u(\tn_n) ) - \int_{t_{n-1}}^{t_n} f(s,u(s)) \diff{s} 
    \Big\|_{L^2(\Omega;\R^d)}^2\\
    & \quad = \Big\| k f(\tn_n, u(\tn_n) )\Big\|_{L^2(\Omega;\R^d)}^2 
	    - 2 \E \Big[ \Big(k f(\tn_n, u(\tn_n) ) , \int_{t_{n-1}}^{t_n} f(s,u(s)) 
	    \diff{s}\Big) \Big]\\
	  & \qquad + \Big| \int_{t_{n-1}}^{t_n} f(s,u(s)) \diff{s} 
	  \Big|^2\\
	  & \quad = \Big\| k f(\tn_n, u(\tn_n) )\Big\|_{L^2(\Omega;\R^d)}^2 
	  - \Big| \int_{t_{n-1}}^{t_n} f(s,u(s)) \diff{s} 
	  \Big|^2\\
    & \quad \le \big\| k f(\tn_n, u(\tn_n) ) \big\|_{L^2(\Omega;\R^d)}^2,
  \end{align*}
  due to \eqref{eq4:mean}. Moreover, since $0 \in K_u \subset \R^d$ we derive
  from Assumption~\ref{as:ODEf} (ii) and (iii) that
  \begin{align*}
    &\big\| k f(\tn_n, u(\tn_n) ) 
    \big\|_{L^2(\Omega;\R^d)} \\
    &\quad \le k \big\| f(\tn_n, 0 ) \big\|_{L^2(\Omega;\R^d)} 
    + k \big\| f(\tn_n, u(\tn_n) ) - f(\tn_n, 0 ) 
    \big\|_{L^2(\Omega;\R^d)}\\
    &\quad \le k \big( \E \big[ | g( \tn_n ) |^2 \big]
    \big)^{\frac{1}{2}} + k \big( \E \big[
    L_{K_u}(\tn_n )^2 | u( \tn_n ) |^2 \big]
    \big)^{\frac{1}{2}}\\
    &\quad \le k^{\frac{1}{2}} \Big( \int_{t_{n-1}}^{t_n} g(s)^2 \diff{s}
    \Big)^{\frac{1}{2}} 
    + k^{\frac{1}{2}} \| u \|_{C([0,T];\R^d)} \Big( \int_{t_{n-1}}^{t_n} 
    L_{K_u}(s)^2
    \diff{s} \Big)^{\frac{1}{2}}. 
  \end{align*}
  In summary, we have shown that
  \begin{align*}
    \big\| \rho_N^n(u|_N) \big\|_{L^2(\Omega;\R^d)}
    &\le  \big( 1 + \|u \|_{C([0,T];\R^d)} \big) \Big(1 + 
    T^{\frac{1}{2}} \big\| L_{K_u} + g  \big\|_{L^2(0,T;\R)} \Big)\\
    &\quad \times \Big( \big\| g \big\|_{L^2(t_{n-1},t_n;\R)} + \big\| L_{K_u}
    \big\|_{L^2(t_{n-1},t_n;\R)} \Big) k^{\frac{1}{2}}.
  \end{align*}  
  This completes the proof of \eqref{eq4:cons1}.
\end{proof}


We are now well-prepared to state and prove the main result of this section.

\begin{theorem}
  \label{th:convRandEuler}
  Let Assumption~\ref{as:ODEf} be satisfied. For $N \in \N$ 
  let $(U^n)_{n \in \{0,\ldots,N\}} \in \mathcal{G}^2_N$ be the grid function
  generated by the randomized backward Euler method \eqref{eq4:RandBackEuler}
  with step size $k = \frac{T}{N}$. If $\nu k < \frac{1}{4}$, then
  there exists a constant $C$ independent of $N$ and $k$ such that
  \begin{align}
    \label{eq4:errODE}
    \max_{n \in \{0,\ldots,N\}} \big\| U^n - u(t_n) 
    \big\|_{L^2(\Omega;\R^d)}   
    \le C k^{\frac{1}{2}}.
  \end{align}
\end{theorem}

\begin{proof}
  Let us fix an arbitrary $N \in \N$ such that $\nu k < \frac{1}{4}$.
  First note that the sequence $(U^n)_{n \in \{0,\ldots,N\}}
  \in \mathcal{G}^2_N$ is well-defined by Lemma~\ref{lem:exist}. Furthermore, 
  as we already discussed above, the restriction $u|_N$ defined in 
  \eqref{eq4:restriction} is also an element of
  $\mathcal{G}^2_N$. Hence, Lemma~\ref{lem:stab} is applicable with $V = u|_N$.
  Using that $U^0 = u_0 = u(t_0)$ we therefore obtain
  \begin{align*}
    &\big\| U^n - u(t_n) \big\|^2_{L^2(\Omega;\R^d)}\\ 
    &\quad\le \ee^{2(2 \nu + 1)t_n}
    \sum_{j = 1}^n \Big( 2 \big\| \rho^j_N(u|_N) 
    \big\|^2_{L^2(\Omega;\R^d)}
    + \frac{2}{k}  \big\| \E \big[ \rho^j_N(u|_N) \, | \, \F_{j-1} \big] 
    \big\|^2_{L^2(\Omega;\R^d)} \Big)
  \end{align*}
  for every $n \in \N$. After taking the maximum over $n \in \{0, \ldots, N\}$
  it remains to estimate the two sums over the local residuals of the exact
  solution. From Lemma~\ref{lem:cons} we get
  \begin{align*}
    2 \sum_{j = 1}^N \big\| \rho^j_N(u|_N) \big\|^2_{L^2(\Omega;\R^d)} 
    &\le C_1 k \sum_{j = 1}^N \big( \| g \|_{L^2(t_{j-1}, t_j; \R)} +
    \| L_{K_u} \|_{L^2(t_{j-1}, t_j; \R)} \big)^2\\
    &\le 2 C_1 k \sum_{j = 1}^N \big( \| g \|_{L^2(t_{j-1}, t_j; \R)}^2 +
    \| L_{K_u} \|_{L^2(t_{j-1}, t_j; \R)}^2 \big)\\
    &= 2 C_1 \big( \| g \|_{L^2(0,T; \R)}^2 +
    \| L_{K_u} \|_{L^2(0, T; \R)}^2 \big) k,
  \end{align*}
  where the constant $C_1$ is given by
  \begin{align*}
    C_1 = 2 \big( 1 + \|u \|_{C([0,T];\R^d)} \big)^2 \Big(1 + T^{\frac{1}{2}}
    \big\| L_{K_u} + g \big\|_{L^2(0,T;\R)} \Big)^2.
  \end{align*}
  In addition, Lemma~\ref{lem:cons} also yields
  \begin{align*}
    &\frac{2}{k} \sum_{j = 1}^n \big\| \E \big[ \rho^j_N(u|_N) \, | \, \F_{j-1}
    \big] \big\|^2_{L^2(\Omega;\R^d)} \\
    &\quad \le C_2 k \sum_{j = 1}^N \| L_{K_u} \|_{L^2(t_{j-1}, t_j;
    \R)}^2 = C_2 \| L_{K_u} \|_{L^2(0, T; \R)}^2 k,
  \end{align*}
  with 
  \begin{align*}
    C_2 = 2 \big( 1 + \|u \|_{C([0,T];\R^d)} \big)^2
    \big\| L_{K_u} + g \big\|_{L^2(0,T;\R)}^2.
  \end{align*}
  Altogether, this proves \eqref{eq4:errODE} with
  \begin{align*}
    C = \ee^{(2 \nu + 1)T} \sqrt{\max(2 C_1, C_2)} \Big( \| g \|_{L^2(0,T;
    \R)} + \| L_{K_u} \|_{L^2(0, T; \R)} \Big).
  \end{align*}
\end{proof}


\section{Numerical experiments for ODEs}
\label{sec:numexpODE}

A simple, yet useful problem to demonstrate the usability of the randomized
backward Euler method \eqref{eq4:RandBackEuler} is the Prothero--Robinson
example from \cite{prothero1974}, see also \cite[Sec.~IV.15]{hairer2010}, 
which is given by
\begin{align}
	\begin{split}
    \label{eq:ProtheroRobinson}
    \begin{cases}
      \dot{u}(t) = \lambda (u(t) - g(t)) + \dot{g}(t), \quad \text{ for almost
      all } t \in (0,T],\\
      u(0) = g(0),
    \end{cases}
	\end{split}
\end{align}
for $\lambda \in \R$ and $g \in H^{1}(0,T)$. Here $H^1(0,T)$ denotes the
standard Sobolev space of square integrable and weakly differentiable
functions.
It is easy to verify that $u = g$ 
is a solution to \eqref{eq:ProtheroRobinson} in the sense of Carath\'{e}odory. 
The right-hand side $f
\colon [0,T] \times \R \to \R$ is given by
\begin{align*}
	f(t,x) := \lambda (x - g(t)) + \dot{g}(t), \quad t \in [0,T], \; x \in \R,
\end{align*}
which fulfills Assumption~\ref{as:ODEf}, as can easily be shown.

For a numerical example, we choose $T = 1$ and
a function $g$ which is oscillating with a 
period $2 p$, for $p = 2^{-K}$, $K\in \N$. To this end, we 
use a continuous, piecewise linear function $g$. This function is chosen 
such that it fulfills
\begin{align*}
	g(i p) = 
	\begin{cases}
		p, \quad &\text{ for } i \in \{ 0,\dots, 2^K \} \text{ odd},\\
		0, \quad &\text{ for } i \in \{ 0,\dots, 2^K \} \text{ even},
	\end{cases}
\end{align*}
and the affine linear interpolation of these values
for all other $t \in [0,T]$.
Further, the function $g$ has a weak derivative in $L^2(0,1)$. 
For the implementation we take the following 
representation for $\dot{g}$ given by 
\begin{align*}
  \dot{g}(t) = 
	\begin{cases}
		-1, \quad &\text{ for } t \in [ip,(i+1) p ), \ i\in \{ 0,\dots, 2^K-1 \} 
		\text{ 
		odd,}\\
		1, \quad &\text{ for } t \in [ip,(i+1) p ),  \ i\in \{ 0,\dots, 2^K-1 \} 
		\text{ even}.
	\end{cases}
\end{align*}

For every equidistant step size $k = 2^{-n} > p $, $n \in \N$
with $n < K$, the classical backward Euler method only 
evaluates the mapping $g$ in the grid points, where $g$ is equal to zero and
where the chosen representation of $\dot{g}$ is equal to $1$. Therefore, for
all such step sizes, the classical backward Euler method cannot distinguish
between the problem \eqref{eq:ProtheroRobinson} and the initial value problem
\begin{align*}
  \begin{cases}
      \dot{v}(t) = \lambda v(t) + 1, \quad \text{ for all } t \in (0,T],\\
      v(0) = g(0).
  \end{cases}
\end{align*}
Since $u = g \neq v$ it is not surprising that the classical backward Euler
method does not yield a good approximation of the correct solution. 
Only for $k<p$ it becomes visible that the classical 
backward Euler method converges to the exact solution $u = g$.

On the other hand, the randomized scheme \eqref{eq4:RandBackEuler}
is not so easily ``fooled'' by the
highly oscillating function $g$. It already yields more reliable results for
step sizes $k > p $, since it evaluates $g$ and $\dot{g}$ not only in extremal
points. In Figure~\ref{fig1} we indeed see that the error of the randomized 
scheme \eqref{eq4:RandBackEuler} measured in the $L^2(\Omega, \R)$-norm is
significantly smaller than that of the classical backward Euler method. 

Obviously, a simple way to correct the backward Euler method would be to 
choose a different temporal grid. For instance, one might use a non-equidistant
partition of $[0,T]$ or an adaptive version of the backward Euler method.
However, no matter what deterministic strategy is used, it is always possible
to construct a similar ``fooling'' function $f \colon [0,T] \times \R \to \R$ 
that satisfies Assumption~\ref{as:ODEf} and deceives the deterministic 
algorithm to approximate the wrong initial value problem for all 
computationally feasible numbers of function evaluations.

A further interesting aspect of problem \eqref{eq:ProtheroRobinson} is the 
fact that for $ \lambda < 0 $ it has a dissipative structure, i.e., there 
exists $\nu \in [0, \infty)$ such that 
\begin{align*}
\big( f(t, x) - f(t,y), x -y \big) \le - \nu | x - y |^2,
\end{align*}
holds for all $x, y \in \R$ and $t \in [0,1]$. 
It is well-known, see the discussions in \cite{hairer2010},
that this structure of the problem can be exploited more
efficiently with an implicit scheme in comparison to explicit Runge--Kutta
methods. Here, we will compare the randomized backward Euler method
\eqref{eq4:RandBackEuler} with its explicit randomized counterpart
\begin{align}
	\label{eq7:RandForwardEuler}
		\begin{split}
			\begin{cases}
				U^n = U^{n-1} + k f( \tn_{n}, U^{n-1} ),\quad \text{ for } 
				n \in \{ 1,\ldots,N \},\\ 
				U^0 = u_0,				
			\end{cases}
		\end{split}
\end{align}
which has been studied in \cite{daun2011, heinrich2008, jentzen2009d,
kruse2017a}. In this particular example, we obtain the scheme
\begin{align*}
	\begin{cases}
		U^n = (1 + k \lambda )U^{n-1} - k \lambda g(\tn_{n})  + k 
		\dot{g}(\tn_{n}),\quad \text{ for } 
		n \in \{ 1,\ldots,N \},\\ 
		U^0 = u_0.
	\end{cases}
\end{align*}
This will lead to an oscillating numerical solution with a high amplitude if 
$|1+ k \lambda| >1$ holds true. For $\lambda < 0$ this is the case if $k < 
-\frac{2}{\lambda}$. 

\begin{figure}[t] 
	\centering
	\includegraphics[width=0.49\textwidth]{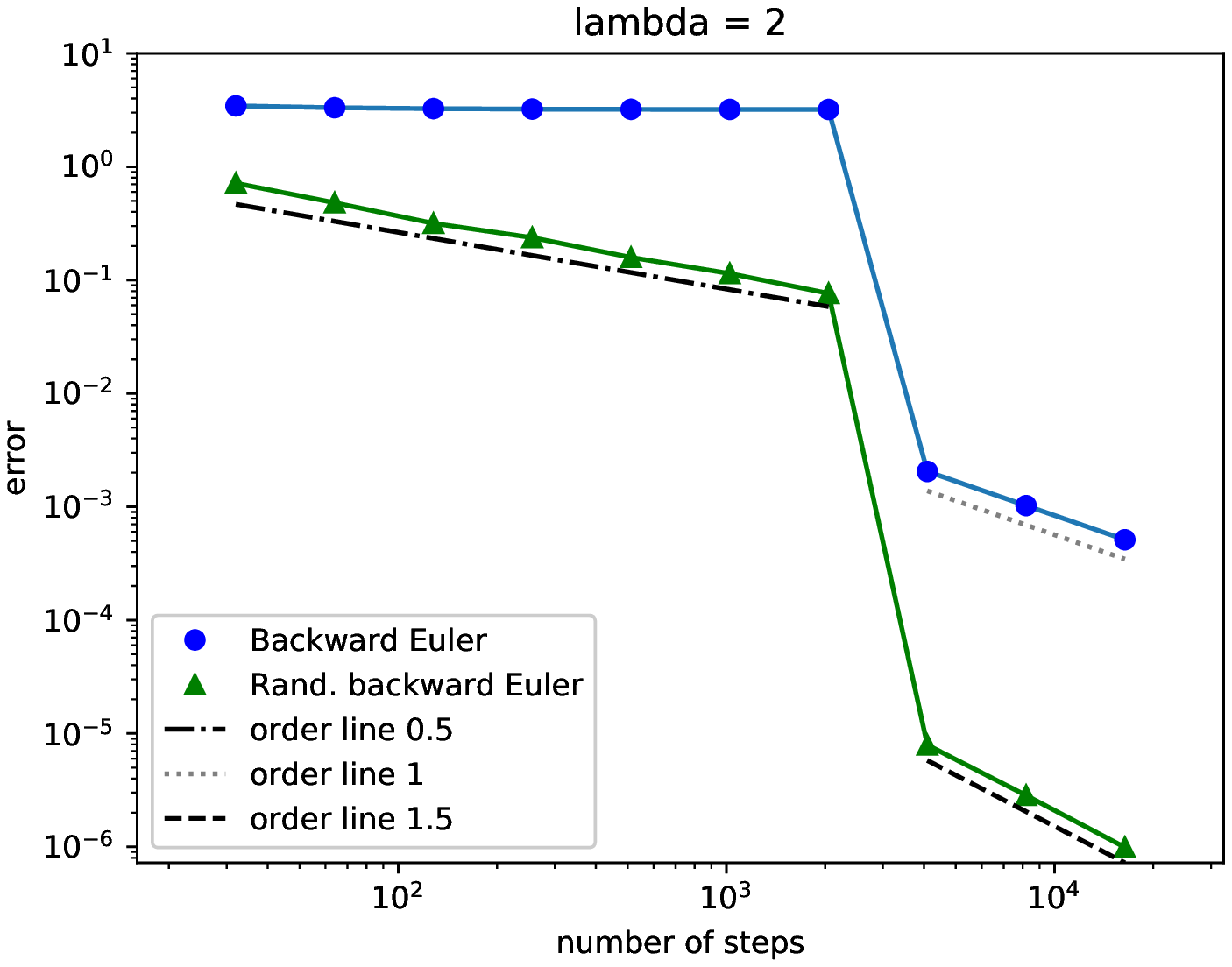}
	\includegraphics[width=0.49\textwidth]{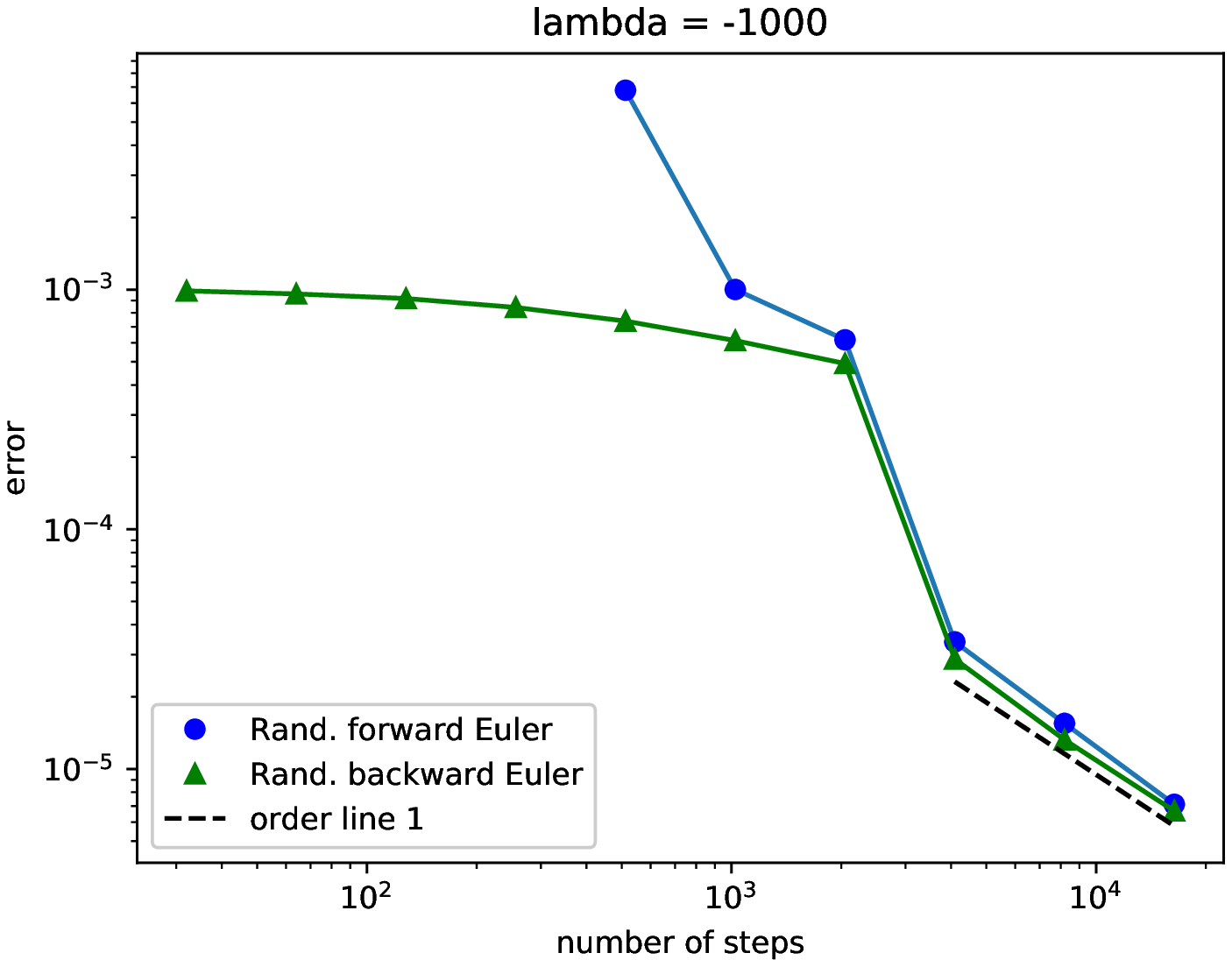}
	\caption{Left: $L^2$-convergence of the 
  classical backward Euler method and scheme \eqref{eq4:RandBackEuler} to the
  IVP \eqref{eq:ProtheroRobinson} with $\lambda =2$.
  Right: $L^2$-convergence of schemes \eqref{eq4:RandBackEuler} and
	\eqref{eq7:RandForwardEuler} to \eqref{eq:ProtheroRobinson} with $\lambda 
		=-1000$.}
	\label{fig1}
\end{figure}

In the numerical examples that lead to Figure~\ref{fig1}, we considered the 
value $p = 2^{-12}$ and step sizes $k = 2^{-n}$ for $n \in \{5, \dots, 14 
\}$. To evaluate the $L^2(\Omega;\R)$-norm we considered $1000$ Monte Carlo 
iterations. 
In the plot on the left hand side we used the value $\lambda = 2$ and compared 
the classical backward Euler method with scheme \eqref{eq4:RandBackEuler}. 
As we expected from the discussion above, two different phases of the example
become well visible. For $n \in \{5, \dots, 11 \}$ the classical 
backward Euler method does not offer an accurate numerical solution. The error
of the randomized backward Euler method decreases with a rate of  
approximately $0.5$. When $n$ changes from $11$ to $12$ both schemes improve 
drastically since they are now able to fully resolve the oscillations of the
solution. In the last part, for $n \in \{12, 13, 14 \}$ the errors of both
schemes decrease with a larger rate. 
Also here, the randomized scheme appears to have a higher rate of convergence, 
$1.5$, than the classical scheme which converges with rate $1$.  
Note that the rate of $1.5$ is in line with those of randomized quadrature 
rules, see \cite{kruse2017a}.

In the plot on the right-hand side in Figure~\ref{fig1}, we considered the 
case
$\lambda =  -1000$ and compared the randomized backward Euler method 
\eqref{eq4:RandBackEuler} with the randomized forward Euler method 
\eqref{eq7:RandForwardEuler}. 
Here, we only plotted errors smaller than $1$, since the explicit scheme 
produces strongly oscillating numerical solutions with a very large amplitude
for step sizes which are not small enough. The first occurring error of the
scheme \eqref{eq7:RandForwardEuler} in the plot  
appears for $2^9 = 512$ temporal steps. This was expected since 
the explicit scheme only leads to a non-exploding solution for step sizes $k$
with $|1 + k \lambda| < 1$.

To sum up, the numerical experiments in this section indicate that the
randomized backward Euler method is especially advantageous compared to
deterministic methods if the problem has very irregular coefficients.
Compared to explicit randomized Runge--Kutta methods such as
\eqref{eq7:RandForwardEuler} we also obtain more reliable results for rather
large step sizes when considering problems with a dissipative structure. Both 
points qualify the scheme \eqref{eq4:RandBackEuler} for
the numerical treatment of monotone evolution equations with time-irregular
coefficients. This will be studied in more details in the following sections.


\section{A non-autonomous nonlinear evolution equation with time-irregular 
coefficients}
\label{sec:nonlinPDE}

In this section, we now turn our attention to the second class of initial value
problems we consider in this paper. More precisely, we are interested in 
non-autonomous and possibly nonlinear evolution equations of the form
\begin{align} 
  \label{eq6:AbstractProb}
  \begin{split}
    \begin{cases}
      \dot{u}(t) + \A(t) u(t) = f(t), \quad \text{for almost all } t \in 
      (0,T],\\
      u(0) = u_0.    
    \end{cases}
  \end{split}
\end{align}
In order to make this rather abstract setting 
more precise, we start by introducing the
real, separable Hilbert spaces $(V,\ska[V]{\cdot}{\cdot}, \|\cdot\|_V)$ and 
$(H,\ska[H]{\cdot}{\cdot}, \|\cdot\|_H)$. Here, we assume that the space $V$ 
is densely embedded in the space $H$. Thus, we obtain the Gelfand triple
\begin{align*}
  V \stackrel{d}{\hookrightarrow} H \cong H^\ast 
  \stackrel{d}{\hookrightarrow} V^\ast,
\end{align*}
where $H^*$ and $V^*$ are the dual spaces of $H$ and $V$, respectively. These 
spaces are equipped with the induced dual norms. 

We impose the following conditions on $\A$. Note that, as it is customary, we
usually write $\A(t)v$ instead of $\A(t,v)$. 

\begin{assumption}
  \label{as:nonlinPDE}
  The mapping $\A \colon [0,T] \times V \to V^\ast$
  fulfills the conditions:
  \begin{enumerate}[label=(\roman*)]
    \item For every $v_1, v_2 \in V$ the mapping $[0,T] \ni t \mapsto 
    \inner[V^*,V]{\A(t) v_1}{v_2}$ is measurable.
    \item \label{item:Abounded} There exists a constant $M\geq 0$ such that 
    $\| \A(t) 0 \|_{V^*} 
      \leq M$ for every $t\in [0,T]$.
    \item \label{item:Alip}
    There exists $L \in (0,\infty)$ such that for all $t \in [0,T]$
      it holds true that
      \begin{align*}
        \| \A(t) v_1 - \A(t) v_2 \|_{V^\ast} \le L \| v_1 - v_2 \|_V, \quad
        \text{ for all } v_1, v_2 \in V.
      \end{align*}
    \item \label{item:Amon}
    There exists $\mu \in (0,\infty)$ such that for all $t \in
      [0,T]$ it holds true that
      \begin{align*}
        \inner[V^\ast,V]{\A(t) v_1 - \A(t) v_2 }{v_1 - v_2} \ge \mu \|v_1 -
        v_2\|^2_V, \quad \text{ for all } v_1, v_2 \in V.         
      \end{align*}
  \end{enumerate}
\end{assumption}

\begin{remark}\label{rem:Gaarding}
  Instead of Assumption~\ref{as:nonlinPDE}~\ref{item:Amon} we can ask for 
  the weaker condition
  \begin{enumerate}[label=(\roman*)]
    \item[$(iv')$]
    There exist $\mu \in (0,\infty)$ and $\kappa \in [0,\infty)$ such that for all 
    $t \in [0,T]$ it holds true that
    \begin{align*}
      \inner[V^\ast,V]{\A(t) v_1 - \A(t) v_2 }{v_1 - v_2} 
      \ge 
      \mu \|v_1 - v_2\|^2_V - \kappa \|v_1 - v_2\|^2_H ,
    \end{align*}
    for all $v_1, v_2 \in V$.
  \end{enumerate}
  Using this G{\aa}rding-type inequality, the following proofs can be done in 
  an analogous manner with a further application of Gronwall's inequality. 
  This additional argument leads to a constant $C$ in 
  Theorem~\ref{th6:conv} below that grows exponentially in time. For 
  simplicity we will only treat the case $\kappa = 0$ in the following. 
\end{remark}

Before we analyze the convergence of the numerical  
scheme \eqref{eq6:parabolicScheme} defined below, let us recall the existence
of a unique solution to the abstract problem \eqref{eq6:AbstractProb}.
We will consider the concept of weak solutions for abstract non-autonomous 
problems of the form \eqref{eq6:AbstractProb}, i.e., we call a function 
\begin{align*}
  u \in \mathcal{W}(0,T) = \big\{ v \in L^2(0,T;V) \, : \,\dot{v} \in 
  L^2(0,T;V^*) 
  \big\} 
\end{align*}
a \emph{weak solution} to \eqref{eq6:AbstractProb} if $u(0)=u_0$ is fulfilled
and if the integral equality
\begin{align}
  \label{eq6:weaksol}
  \int_{0}^{T} \inner[V^*,V]{\dot{u}(t) + \A(t) u(t)}{v(t)} \diff{t} = 
  \int_{0}^{T} \inner[V^*,V]{f(t)}{v(t)} \diff{t}
\end{align}
is satisfied for every $v\in L^2(0,T;V)$. Note that evaluating
the abstract function $u$ at the initial time is well defined since
the space $\mathcal{W}(0,T)$ is embedded in the space $C([0,T];H)$. An
introduction to this concept of solutions can be found in, for example,
\cite{emmrich2004}, \cite{evans1998} or \cite{roubicek.2013}.

\begin{prop} \label{prop:exPDE}
  Let Assumption~\ref{as:nonlinPDE} be satisfied. Then for every
  given $f\in L^2(0,T;H)$ and initial value $u_0\in H$ there
  exists a unique weak solution $u \in \mathcal{W}(0,T)$ to the problem 
  \eqref{eq6:AbstractProb}. 
\end{prop}

Most proofs for this kind of statement that can be found in the literature 
are either for linear problems, see for example \cite[Cor.~23.26]{zeidler1990} 
or \cite[Satz~8.3.6]{emmrich2004}, or for nonlinear 
problems in a Browder--Minty setting, compare for example 
\cite[Thm~30.A]{zeidler1990B}, \cite[Satz~8.4.2]{emmrich2004} or 
\cite[Theorem~8.9]{roubicek.2013}. Our assertion 
is intermediate since we consider nonlinear operators that are still 
Lipschitz continuous. Therefore, the aforementioned references for nonlinear 
problems can be used but we note that also small modifications of the proofs 
for linear problems would be sufficient. 

\begin{remark}
	Note that for mere existence results, it is sufficient to assume $f \in 
	L^2(0,T;V^*)+ L^1(0,T;H)$. The last proposition and some of the following 
	statements would also hold under this more general condition. To obtain a 
	rate of convergence for the numerical scheme, the additional assumption 
	$f\in L^2(0,T;H)$ will be essential.
\end{remark}

In the following, we will consider a full discretization of the
problem \eqref{eq6:AbstractProb},  
i.e., we will discretize the equation both in time and space. For this purpose
let $N \in \N$ denote the number of temporal steps and set $k =\frac{T}{N}$ as 
the temporal step size. For this particular $N$ and $k$ we obtain an 
equidistant partition of the interval $[0,T]$ given by $t_n := k n$, $n \in 
\{0,\dots,N\}$. 
Further, we introduce the family of independent and 
$\mathcal{U}(0,1)$-distributed random variables $\tau = (\tau_{n})_{n \in \N}$ 
on a complete probability space $(\Omega,\F,\P)$ and write $\tn_n = t_n
+  k\tau_n$ for $n\in\N$. Let $(\F_n)_{n\in \{0,\dots, N\}}$ be the complete
filtration which is induced by $(\tn_n)_{n\in \{0,\dots, N\}}$, compare
with  \eqref{eq3:filtration}.

For the space discretization we consider an abstract Galerkin method.
To this end let $(V_h)_{h \in (0,1)}$ be a sequence of
finite-dimensional  subspaces of $V$ each endowed with the inner product 
$\ska[H]{\cdot}{\cdot}$ and the norm $\|\cdot\|_H$ of $H$. Further, for each 
$h \in (0,1)$ 
we denote by $P_h \colon H \to V_h$ the orthogonal
projection onto the Galerkin space $V_h$ with respect to $(\cdot,\cdot)_H$.
More precisely, for each $v \in H$ we define $P_h v$ as the uniquely
determined element in $V_h$ that satisfies
\begin{align}
  \label{eq6:defPh}
  ( P_h v, w_h )_H = (v, w_h)_H, \quad \text{ for all } w_h \in V_h.
\end{align}
In order to formulate the equation \eqref{eq6:AbstractProb} in a suitable 
discrete setting, we also introduce a discrete version $\A_h \colon [0,T] 
\times V_h
\to  V_h$ of the operator $\A$. This is accomplished in the same way as above
by defining $\A_h(t) v_h$ for given
$t\in [0,T]$ and $v_h \in V_h$ as the unique element in $V_h$ that fulfills
\begin{align}
  \label{eq6:defAh}
  \big( \A_h(t) v_h, w_h \big)_H = \inner[V^*,V]{\A(t) v_h }{w_h}
\end{align}
for every $w_h \in V_h$. The existence of a unique $\A_h(t) v_h \in V_h$ 
follows directly from the Riesz representation theorem.

Our aim is to examine the numerical scheme
\begin{align}
  \label{eq6:parabolicScheme}
  \begin{split}
	  \begin{cases}
	    U^{n}_{h} + k \A_h(\tn_{n}) U^{n}_{h} = k P_h f(\tn_n) + 
	    U^{n-1}_{h},\quad \text{ for } 
	    n \in \{ 1,\ldots,N \},\\
	    U^{0}_h = P_h u_0.  
	  \end{cases}
  \end{split}
\end{align}
Note that, as in the finite-dimensional case in Section~\ref{sec:randEul},
the numerical approximation $(U^n_h)_{n\in \{0,\dots, N \}}$ consists of a
family of random variables taking values in $V_h$.
Before we analyze the convergence of the scheme \eqref{eq6:parabolicScheme}
the following lemma shows that $(U^n_h)_{n \in \{0,\ldots,N\}}$ is indeed
well-defined for every value of the step size $k$.

\begin{lemma}
	\label{lem6:exist}
  Let Assumption~\ref{as:nonlinPDE} be satisfied. Then for every inhomogeneity 
  $f\in L^2(0,T;H)$, every initial value $u_0\in H$, and every step size $k =
  \frac{T}{N}$, $N \in \N$, there exists a unique  
  solution  $(U^n_h)_{n\in  \{0,\dots,N \}}$ to the implicit scheme 
  \eqref{eq6:parabolicScheme} such that for every $n \in \{1,\ldots,N\}$ 
  the element $U^n_h$ is $\F_n$-measurable and $U^n_h(\omega) \in V_h$ for 
  almost every $\omega \in \Omega$.
\end{lemma}

\begin{proof}
  Let $h \in (0,1)$ be fixed. To prove the existence of a suitable solution
  to \eqref{eq6:parabolicScheme}, we introduce an equivalent problem in $\R^d$
  with $d = \dim(V_h)$ such that we can apply arguments from
  Section~\ref{sec:randEul} to prove the existence of a unique solution
  $(U^n_h)_{n\in \{0,\dots,N \}}$. 
  To this end, we consider a basis $\{\psi_1,\dots, \psi_d\}$ of
  the finite-dimensional  space $V_h$  
  and test \eqref{eq6:parabolicScheme} with a basis element 
  $\psi_j$, $j \in \{ 1,\dots,d \}$. Then \eqref{eq6:parabolicScheme} can 
  equivalently be rewritten as the following system of scalar equations
  \begin{align}
    \label{eq6:testeq}
    \begin{split}
	    \begin{cases}
	      \ska[H]{U^n_h + k \A_h(\tn_n) U^n_h}{\psi_j} 
	      = \ska[H]{k P_h f(\tn_n) + U^{n-1}_h}{\psi_j}, \quad \text{ for } 
	      n \in \{ 1,\ldots,N \}, \\
	      ( U_h^0, \psi_j )_H = ( u_0, \psi_j )_H, 
	    \end{cases}
    \end{split}
  \end{align}
  for all $j \in \{ 1,\ldots, d \}$.
  Since the inhomogeneity $P_h f \in L^2(0,T;V_h)$ takes
  values in $V_h$ it can be represented by
  \begin{align*}
    P_h f = \sum_{i=1}^{d} f_{h,i} \psi_i,
  \end{align*}
  where $f_{h,i} \in L^2(0,T; \R)$ for each $i \in \{ 1, \ldots, d \}$. 
  In order to prove the existence of the $V_h$-valued random variable $U^n_h$,
  we will show that there exist measurable functions $\alpha^n_{h,i} \colon
  \Omega \to \R$, $i \in \{ 1,\ldots,d\}$, $n \in \{0,\ldots,N\}$, such that  
  \begin{align*}
    U^n_h = \sum_{i=1}^{d} \alpha_{h,i}^{n} \psi_i
  \end{align*}
  satisfies \eqref{eq6:parabolicScheme}. For $n= 0$ this follows at once.
  
  For the case $n > 0$ let us denote the vector of all coordinates
  $(\alpha^n_{h,i})_{i \in \{1,\ldots,d\}}$
  and $(f_{h,i})_{i \in \{1,\ldots,d\}}$  by 
  \begin{align*}
    \mathbf{u_{h}^{n}}(\omega) := \big (\alpha_{h,i}^{n}(\omega) 
    \big)_{i \in \{ 1,\dots,d \}}
    \quad \text{and}\quad
    \mathbf{f_h}(t) := \big( f_{h,i}(t) \big)_{i \in \{ 1,\dots, d \}}
  \end{align*}
  for almost every $\omega \in \Omega$ and $t\in [0,T]$.  
  Furthermore, we denote the mass matrix in $\R^{d,d}$ by
  \begin{align*}
    \mathbf{M_h} = \big( \ska[H]{\psi_i}{\psi_j} \big)_{i,j \in \{ 1,\dots, 
    d\}}.
  \end{align*}
  It is easily seen that $\mathbf{M_h} \in \R^{d,d}$
  is symmetric and positive definite. 
  In order to obtain a corresponding representation for $\A_h(t) \colon 
  V_h \to V_h$, $t\in [0,T]$, we introduce $\mathbf{A_h} \colon [0,T] \times
  \R^d \to \R^d$ such that for $t\in [0,T]$ and $\mathbf{x} \in \R^d$ 
  the vector $\mathbf{A_h}(t,\mathbf{x}) \in \R^d$ is determined by
  \begin{align*}
    \sum_{i=1}^{d} \big[ \mathbf{A_h}(t,\mathbf{x})\big]_i \psi_i
    = \A_h(t)v_{\mathbf{x}} \in V_h,
  \end{align*}
  where $v_{\mathbf{x}} = \sum_{i=1}^{d} \mathbf{x}_i \psi_i$.
  Then \eqref{eq6:testeq} can equivalently be written as
  \begin{align*}
    \mathbf{M_h}\mathbf{u^n_h} + k \mathbf{M_h} \mathbf{A_h}(\tn_n, 
    \mathbf{u^n_h}) 
    = k \mathbf{M_h} \mathbf{f_h}(\tn_n) + \mathbf{M_h} \mathbf{u^{n-1}_h}, 
  \end{align*}
  or simply
  \begin{align*}
    \mathbf{u^n_h}  = \mathbf{u^{n-1}_h} 
    + k \big(\mathbf{f_h}(\tn_n) - \mathbf{A_h}(\tn_n,  \mathbf{u^n_h})\big).
  \end{align*}
  
  In order to transfer the monotonicity and Lipschitz continuity of $\A_h$ to 
  its counterpart, we introduce the following inner product and norm in 
  $\R^d$: 
  \begin{align*}
  \ska[\mathbf{M_h}]{\mathbf{x}}{\mathbf{y}} = \mathbf{x}^T 
  \mathbf{M_h}\mathbf{y}
  \quad \text{ and } \quad
  \|\mathbf{x}\|_{\mathbf{M_h}} = 
  \sqrt{\ska[\mathbf{M_h}]{\mathbf{x}}{\mathbf{x}}}
  \end{align*}
  for $\mathbf{x}, \mathbf{y} \in \R^d$. This particular choice of inner 
  product coincides with the inner product of $H$ of the elements 
  $u_{\mathbf{x}}, v_{\mathbf{x}} \in H$ given by
  \begin{align*}
  u_{\mathbf{x}} = \sum_{i=1}^{d} \mathbf{x}_i \psi_i
  \quad \text{ and } \quad
  v_{\mathbf{y}} = \sum_{i=1}^{d} \mathbf{y}_i \psi_i
  \end{align*}
  for $\mathbf{x},\mathbf{y} \in \R^d$, i.e., the following equalities hold:
  \begin{align*}
  \ska[\mathbf{M_h}]{\mathbf{x}}{\mathbf{y} } = 
  \ska[H]{u_{\mathbf{x}}}{v_{\mathbf{y}}}
  \quad \text{and} \quad
  \|\mathbf{x}\|_{\mathbf{M_h}} = \|u_{\mathbf{x}} \|_H.
  \end{align*}

  To prove the existence of an element $\mathbf{u^n_h}(\omega)$ for almost
  every $\omega \in \Omega$ we use Lemma~\ref{lem:root}. To this end, we 
  introduce the function
  \begin{align*}
    \mathbf{g} \colon \Omega \times \R^d \to \R^d, \quad 
    \mathbf{g}(\omega,\mathbf{x}) = 
    \mathbf{x} - \mathbf{u^{n-1}_h}(\omega) 
    - k \big(\mathbf{f_h}(\tn_n(\omega)) - \mathbf{A_h}(\tn_n(\omega),  
    \mathbf{x})\big).
  \end{align*}
  Define $q = q(\omega) := \| U^{n-1}_h (\omega)\|_{H} + k \| 
  f(\tn_n(\omega))\|_{H}$. Observe that for almost every $\omega \in \Omega$ we
  have $q(\omega) \in [0,\infty)$. In the following we consider an arbitrary
  but fixed $\omega \in \Omega$ with this property. Next, we introduce
  \begin{align*}
    R = R(\omega) = \frac{q}{2} 
      + \sqrt{ \frac{q^2}{4}  + k \frac{M^2}{4\mu}} \in (0, \infty).
  \end{align*}  
  For all $\mathbf{x} \in \R^d$ with $\|\mathbf{x}\|_{\mathbf{M_h}} = R$ we
  then obtain
  \begin{align}
    \notag &\ska[\mathbf{M_h}]{\mathbf{g}(\omega,\mathbf{x})}{\mathbf{x}}\\
    \notag &\quad = \ska[\mathbf{M_h}]{\mathbf{x} - \mathbf{u^{n-1}_h}(\omega) 
      - k \big(\mathbf{f_h}(\tn_n(\omega)) - \mathbf{A_h}(\tn_n(\omega),  
      \mathbf{x})\big)}{\mathbf{x}}\\
			\notag &\quad = \| \mathbf{x} \|^2_{\mathbf{M_h}}
			- \ska[\mathbf{M_h}]{\mathbf{u^{n-1}_h}(\omega)}{\mathbf{x}}
      - k \ska[\mathbf{M_h}]{ \mathbf{f_h}(\tn_n(\omega))}{\mathbf{x}}
      + k \ska[\mathbf{M_h}]{ \mathbf{A_h}(\tn_n(\omega), 
      \mathbf{x})}{\mathbf{x}}\\
    \label{eq6:root1} &\quad \geq R^2
      - R \| U^{n-1}_h (\omega)\|_{H} 
      - k R \| f(\tn_n(\omega))\|_{H} 
      + k \inner[V^*,V]{ \A(\tn_n(\omega)) v_{\mathbf{x}}}{v_{\mathbf{x}}}
  \end{align}
  where $v_{\mathbf{x}} = \sum_{i=1}^{d} \mathbf{x}_i \psi_i$. Using 
  Assumption~\ref{as:nonlinPDE}~\ref{item:Abounded} and~\ref{item:Amon}, the 
  last summand of \eqref{eq6:root1} can be estimated by
  \begin{align}
    \label{eq6:compA}
    \begin{split}
      &\inner[V^*,V]{ \A(\tn_n(\omega)) v_{\mathbf{x}}}{v_{\mathbf{x}}}\\
      &\quad = \inner[V^*,V]{ \A(\tn_n(\omega)) v_{\mathbf{x}} - 
      \A(\tn_n(\omega)) 
      0}{v_{\mathbf{x}} - 0 } 
      + \inner[V^*,V]{ \A(\tn_n(\omega)) 0}{v_{\mathbf{x}} }\\
      &\quad \geq \mu \|v_{\mathbf{x}} \|_{V}^2  
      - \|\A(\tn_n(\omega)) 0 \|_{V^*} \|v_{\mathbf{x}} \|_V\\
      &\quad \geq \mu \|v_{\mathbf{x}} \|_{V}^2  - M \|v_{\mathbf{x}} 
      \|_V\\
      &\quad \geq \mu \|v_{\mathbf{x}} \|_{V}^2  - \frac{M^2}{4 \mu}  - \mu 
      \|v_{\mathbf{x}} \|_V^2 = - \frac{M^2}{4\mu}.     
    \end{split}
  \end{align}
  Therefore, after inserting $R$ we obtain
  \begin{align*}
    \ska[\mathbf{M_h}]{\mathbf{g}(\omega,\mathbf{x})}{\mathbf{x}}
      \geq  R^2
      - R \big( \| U^{n-1}_h (\omega)\|_{H} 
      + k \| f(\tn_n(\omega))\|_{H} \big) 
      - k \frac{M^2}{4\mu}
      = 0.
  \end{align*}
  Since $\mathbf{A_h}(\tn_n(\omega), \cdot )$ is continuous in the second 
  argument due to Assumption~\ref{as:nonlinPDE}~\ref{item:Alip}, this allows 
  us to apply Lemma~\ref{lem:root} and Remark~\ref{rem:InnerProd}. Thus, for 
  almost every $\omega \in \Omega$ we 
  obtain the existence of an element  $\mathbf{x} = \mathbf{x}(\omega) 
  \in \R^d$ such that 
  $\mathbf{g}(\omega, \mathbf{x}) = 0$ holds. To prove that 
  this root is unique, assume that there exist $\mathbf{x}, \mathbf{y} \in 
  \R^d$ such that
  \begin{align*}
    \mathbf{g}(\omega,\mathbf{x})=0 \quad \text{ and } \quad
    \mathbf{g}(\omega,\mathbf{y})=0
  \end{align*}
  is fulfilled. Then, inserting the definition of the function $\mathbf{g}$ 
  leads to
  \begin{align*}
    \| \mathbf{x} - \mathbf{y}\|_{\mathbf{M_h}}^2
    & = - k\ska[\mathbf{M_h}]{\mathbf{A_h}(\tn_n(\omega),\mathbf{x}) - 
    \mathbf{A_h}(\tn_n(\omega),\mathbf{y})}{\mathbf{x} - \mathbf{y}}\\
    & = - k\ska[H]{\A_h(\tn_n(\omega))v_\mathbf{x} - 
      \A_h(\tn_n(\omega)) v_\mathbf{y} }{v_\mathbf{x} - v_\mathbf{y}} \le 0. 
  \end{align*}
  This implies $\mathbf{x} = \mathbf{y}$.
  An application of Lemma~\ref{lem:measurable} then yields that the 
  mapping
  \begin{align*}
    \mathbf{u_h^n} \colon \Omega \to \R^d, \quad \omega \mapsto 
    \mathbf{u_h^n}(\omega) = \big (\alpha_{h,i}^{n}(\omega) 
    \big)_{i \in \{ 1,\dots,d \}}
  \end{align*}  
  is $\F_n$-measurable, where $\mathbf{u_h^n}(\omega) = \mathbf{x}(\omega)$
  is the unique root of $\mathbf{g}(\omega,\cdot)$. To sum up,
  \begin{align*}
    U^n_h = \sum_{i=1}^{d} \alpha_{h,i}^{n} \psi_i,\quad n\in \{ 0,\dots,N\}, 
  \end{align*}
  is the well-defined 
  solution to the scheme \eqref{eq6:parabolicScheme}. Since  $\{ 
  \psi_1,\dots, \psi_d \}$ is a basis of $V_h$ this implies that 
  $U^n_h(\omega) \in V_h$ for almost 
  every $\omega \in \Omega$ and all $n\in \{ 0,\dots,N \}$.
\end{proof}

\begin{lemma}
  \label{lem:aprioriPDE}
  Let Assumption~\ref{as:nonlinPDE} be satisfied. Then for every 
  $f\in L^2(0,T;H)$, every initial value $U_h^0 \in V_h$, and every step size
  $k = \frac{T}{N}$, $N \in \N$, the unique  
  solution  $(U^n_h)_{n\in  \{0,\dots,N \}}$ to the implicit scheme 
  \eqref{eq6:parabolicScheme} satisfies the a priori bound    
  \begin{align*}
    &\max_{n \in\{0, \ldots,N\}} \E \big[\|U_h^n\|_H^2 \big] 
    + \sum_{j=1}^{N} \E \big[\|U_h^j - U_h^{j-1}\|_H^2 \big] 
    + k \mu \sum_{j=1}^{N} \E \big[\|U_h^j \|_V^2 \big]\\
    &\qquad \leq C \big( T + \|U_h^0\|_H^2 + \| f \|_{L^2(0,T;H)}^2 \big),
  \end{align*}
  where the constant $C$ only depends on $M$, $\mu$, and the embedding $V
  \hookrightarrow H$.
\end{lemma}

\begin{proof}
  Due to the definition of scheme \eqref{eq6:parabolicScheme} we can write 
  for every $j \in \{ 1,  \ldots,N\}$
  \begin{align*}
    \frac{1}{k} \big(U_h^j - U_h^{j-1}\big) + \A_h(\tn_j)U_h^j = f(\tn_j).
  \end{align*}  
  We test this equation with $U_h^j$ in the $H$ inner product and 
  apply the polarization identity
  \begin{align*}
    \frac{1}{2}\big(\|U_h^j\|_H^2 - \|U_h^{j-1}\|^2_H + \|U_h^j - 
    U_h^{j-1}\|_H^2 \big) 
    = \ska[H]{U_h^j - U_h^{j-1}}{U_h^j}.
  \end{align*}
  In addition, recall from \eqref{eq6:defAh} and \eqref{eq6:compA} that
  \begin{align*}
    \big( \A_h(\tn_j) U^j_h, U^j_h \big)_H 
    = \big\langle \A(\tn_j) U^j_h, U^j_h \big\rangle_{V^\ast,V}
    \ge \mu \big\| U_h^j \big\|_V^2 - \big\| \A(\tn_j)0 \big\|_{V^\ast} \big\|
    U_h^j \big\|_V.
  \end{align*}
  From this and \eqref{eq6:parabolicScheme} as well as from
  Assumption~\ref{as:nonlinPDE}~\ref{item:Abounded} we obtain that
  \begin{align*}
    &\frac{1}{2k}\big(\|U_h^j\|_H^2 - \|U_h^{j-1}\|_H^2 +     \|U_h^j - 
    U_h^{j-1}\|_H^2 \big) + \mu \|U_h^j\|_V^2\\
    &\quad\leq \frac{1}{k} \ska[H]{U_h^j - U_h^{j-1} }{U_h^j}
    + \ska[H]{\A_h(\tn_j)U_h^j }{U_h^j} 
    +\big\| \A(\tn_j)0 \big\|_{V^\ast} \big\| U_h^j \big\|_V\\
    &\quad = \ska[H]{ f(\tn_j)}{ U_h^j} + M \big\| U_h^j \big\|_V\\
    &\quad \leq \| f(\tn_j) \|_{H} \| U_h^j \|_H + M \| U_h^j \|_V\\
    &\quad \le C \big( 1 +  \| f(\tn_j) \|_{H}^2 \big) + \frac{\mu}{2}
    \| U_h^j \|_V^2,
  \end{align*}
  where the constant $C$ only depends on $M$, $\mu$, and the embedding $V
  \hookrightarrow H$. Next we sum up with respect to $j$ from $1$ to $n$ and 
  obtain
  \begin{align*}
     \|U_h^n\|_H^2 + \sum_{j=1}^{n} \|U_h^j - 
      U_h^{j-1}\|_H^2  + k \mu \sum_{j=1}^{n} \|U_h^j \|_V^2
    \leq \|U_h^0\|_H^2 + k 2 C \sum_{j=1}^{n}\big( 1 +  \| f(\tn_j)
    \|_{H}^2\big).
  \end{align*} 
  Taking the expectation, we further obtain
  for the term containing the inhomogeneity $f$ that
  \begin{align*}
    \E \Big[ k \sum_{j=1}^{n} \| f(\tn_j) \|_{H}^2  \Big]
    = \sum_{j=1}^{n} \int_{t_{j-1}}^{t_j} \| f(t) \|_{H}^2 \diff{t} 
    \leq \big\| f \big\|_{L^2(0,T;H)}^2
  \end{align*}
  holds. This completes the proof.
\end{proof}

After these preparatory results we can now state the abstract convergence
result for the numerical method \eqref{eq6:parabolicScheme}. For its
formulation we define for each $v \in H$
\begin{align}
  \label{eq6:distH}
  \mathrm{dist}_H(v, V_h) &:= \inf_{v_h \in V_h} \| v - v_h \|_H.
\end{align}
Similarly, if $v \in V$ we set
\begin{align}
  \label{eq6:distV} 
  \mathrm{dist}_V(v, V_h) &:= \inf_{v_h \in V_h} \| v - v_h \|_V.
\end{align}
With \eqref{eq6:distH} we therefore measure how well a given element $v \in H$
can be approximated by elements from $V_h$. Since $V_h$ is
finite-dimensional  it is clear that $P_h v \in V_h$ has the best approximation
properties with respect to the $H$-norm, that is
\begin{align}
  \label{eq6:bestH}
  \| P_h v - v \|_H = \mathrm{dist}_H(v, V_h),\quad \text{ for all } v \in H.
\end{align}
In the same way, if we define $Q_h \colon V \to V_h$ as the orthogonal projection
onto $V_h$ with respect to the inner product $(\cdot, \cdot)_V$, then it holds
true that
\begin{align}
  \label{eq6:bestV}
  \| Q_h v - v \|_V = \mathrm{dist}_V(v, V_h),\quad \text{ for all } v \in V.
\end{align}
Since we consider a general Galerkin method in this section we will not
quantify the best approximation property of $(V_h)_{h \in (0,1)}$ at 
this point.

We also mention that the error estimate in Theorem~\ref{th6:conv}
requires the boundedness of $\| P_h \|_{\L(V)}$. However, one cannot expect 
in general that $\sup_{h \in (0,1]} \|P_h \|_{\L(V)} < \infty$. For a
discussion of the stability of the orthogonal projector $P_h$ in case of the
Galerkin finite element method we refer to \cite{bank2014, carstensen2002,
carstensen2004, crouzeix1987}.

\begin{theorem}
  \label{th6:conv}
  Let Assumption~\ref{as:nonlinPDE} be satisfied. Then for a given
  inhomogeneity $f\in L^2(0,T;H)$ and initial value $u_0\in V$ let $u$ be the
  unique weak solution to the abstract problem \eqref{eq6:AbstractProb}. In
  addition, we assume that there exists $\gamma \in (0,1)$ with 
  \begin{align}
		\label{eq6:Hoeldercond}
    u \in C^{\gamma}( [ 0,T]; V) 
  \end{align}
  as well as 
  \begin{align}
    \label{eq6:L2cond}
    \int_0^T \| \A(t) u(t) \|^2_H \diff{t}  < \infty.
  \end{align}
  Then there exists a constant $C$ only depending on $L$,
  $\mu$, and $T$ such that for every step size $k = \frac{T}{N}$, $N \in \N$,
  and $h \in (0,1)$ we have 
  \begin{align*}
    &\max_{n\in \{0,\dots, N\} } \| U^n_h - u(t_{n})  \|_{L^2(\Omega;H)}
    + \Big( k \sum_{n = 1}^N \| U^n_h - u(t_{n})  \|_{L^2(\Omega;V)}^2
    \Big)^{\frac{1}{2}}
    \\
    &\quad \le C \Big( k^{\frac{1}{2}} \big(\| f \|_{L^2(0,T;H)} + 
    \|\A(\cdot) u (\cdot) \|_{L^2(0,T;H)} \big) + k^{\gamma} \| u
    \|_{C^{\gamma}( [0,T];V)}\\
    &\qquad +\max_{n\in \{0,\dots, N\} } \mathrm{dist}_H(u(t_n), V_h)
    + \big( 1 + \| P_h \|_{\L(V)} \big)
    \Big( k \sum_{n = 1}^N  \mathrm{dist}_V(u(t_n), V_h)^2
    \Big)^{\frac{1}{2}} \Big),
  \end{align*}
  where $(U^n_h)_{n \in \{0,\ldots,N\}} \subset L^2(\Omega;V_h)$ is given by  
  the scheme \eqref{eq6:parabolicScheme}.
\end{theorem}

\begin{proof}
  Throughout the proof we consider an arbitrary but fixed finite-dimensional 
  subspace $V_h$, $h \in (0,1)$, of $V$. Moreover, we denote the error of the
  scheme \eqref{eq6:parabolicScheme} at the time $t_{n}$ by $E^n$, i.e., $E^n 
  := U^n_h - u(t_{n})$ for each $n \in \{ 0,\dots,N\}$. Note that for every $n 
  \ge 1$ we have $E^n \in L^2(\Omega;V)$ since $U^n_h \in L^2(\Omega;V_h)$ by
	Lemma~\ref{lem6:exist} and Lemma~\ref{lem:aprioriPDE}. In addition, due
	to \eqref{eq6:Hoeldercond} we have $u(t) \in V$ for every $t \in [0,T]$.
  
  In the first step, we split the error into two parts using the orthogonal
  projection $P_h \colon H \to V_h$ by 
  \begin{align*}
    E^n = P_h E^n + (I-P_h) E^n =: \Theta^n + \Xi^n.
  \end{align*}
  Due to the orthogonality of $P_h$ with respect to the inner product in $H$ 
  we have
  \begin{align*}
    \|E^n\|^2_H = \|\Theta^n\|^2_H + \| \Xi^n \|_H^2
  \end{align*}
  for every $n \in \{0, \ldots,N\}$. By taking note of \eqref{eq6:bestH} we 
  obtain
  \begin{align*}
    \| \Xi^n \|_H = \mathrm{dist}_H(u(t_n), V_h)
  \end{align*}
  since $\Xi^n = (I-P_h)E^n = (P_h - I) u(t_n)$.
  In addition, we have
  \begin{align*}
    \Big( k \sum_{n = 1}^N \big\| E^n \big\|_{L^2(\Omega;V)}^2
    \Big)^{\frac{1}{2}}
    \le \Big( k \sum_{n = 1}^N \big\| \Theta^n \big\|_{L^2(\Omega;V)}^2
    \Big)^{\frac{1}{2}}
    + \Big( k \sum_{n = 1}^N \big\| \Xi^n \big\|_V^2 \Big)^{\frac{1}{2}}
  \end{align*}
  since $\Xi^n = (P_h - I) u(t_n)$ is deterministic.
  After adding and subtracting the orthogonal projector $Q_h \colon V \to V_h$
  we further obtain the estimate 
  \begin{align*}
    \big\| \Xi^n \big\|_V
    &= \big\| (P_h - I) u(t_n) \big\|_V \le \big\| P_h ( I - Q_h) u(t_n)
    \big\|_V  + \big\| ( Q_h - I) u(t_n) \big\|_V \\
    &\le \big( 1 +  \| P_h \|_{\L(V)} \big) \| (Q_h - I)
    u(t_n)\| = \big( 1 + \| P_h \|_{\L(V)} \big) \mathrm{dist}_V(u(t_n), V_h),
  \end{align*}
  due to \eqref{eq6:bestV}. This shows that
  \begin{align}
    \label{eq6:estXIL2}
    k \sum_{n = 1}^N \big\| \Xi^n \big\|_V^2 
    \le \big( 1 + \| P_h \|_{\L(V)} \big)^2
    k \sum_{n = 1}^N  \mathrm{dist}_V(u(t_n), V_h)^2.
  \end{align}
  Thus it remains to estimate $\E [ \|\Theta^n\|^2_H ]$ and $k \sum_{n = 1}^N
  \E [ \|\Theta^n\|^2_V ] $. To this end, we apply the polarization identity 
  \begin{align}
    \label{eq6:polar}
    \frac{1}{2} \big( \|\Theta^n \|^2_{H} - \|\Theta^{n-1} \|^2_{H} 
    + \|\Theta^n - \Theta^{n-1} \|^2_{H} \big)
    = \ska[H]{\Theta^n - \Theta^{n-1} }{\Theta^n},
  \end{align}
  which holds for every $n \in \{1,\ldots,N\}$. 
  From the orthogonality of $P_h$ with respect to the inner product in $H$ 
  we further have
  \begin{align*}
    \ska[H]{\Theta^n - \Theta^{n-1} }{\Theta^n}
    =\ska[H]{E^n - E^{n-1} }{\Theta^n}
  \end{align*}
  which motivates us to consider the term $E^n - E^{n-1}$ tested with
  $\Theta^n$ in what follows. 

  To estimate the difference of the errors $E^n - E^{n-1} = U^n_h -  u(t_{n}) -
  U^{n-1}_h +  u(t_{n-1})$ we insert the definition of the scheme
  \eqref{eq6:parabolicScheme} and \eqref{eq6:defPh}. This yields 
  \begin{align*}
    \big( U^n_h -  U^{n-1}_h, \Theta^n \big)_H
    &= k\big(  P_h f(\tn_n) - \A_h(\tn_n) U^n_h, \Theta^n \big)_H\\
    &= k \big( f(\tn_n), \Theta^n \big)_H 
    - k \big\langle \A(\tn_n) U^n_h, \Theta^n \big\rangle_{V^\ast,V}.
  \end{align*}
  Moreover, since the random variable $\Theta^n$ takes values in 
  $V_h \subset V$ we get from the canonical embedding $H \cong H^\ast
  \hookrightarrow V^\ast$ and \eqref{eq6:weaksol} that
  \begin{align*}
    \big( u(t_{n}) - u(t_{n-1}), \Theta^n \big)_H
    &= \big\langle u(t_{n}) - u(t_{n-1}), \Theta^n \big\rangle_{V^\ast,V}\\
    &= \int_{t_{n-1}}^{t_n} \big\langle \dot{u}(s),
    \Theta^n \big\rangle_{V^\ast,V} \diff{s}\\ 
    &= \int_{t_{n-1}}^{t_n} \big\langle 
    f(s) - \A(s) u(s), \Theta^n \big\rangle_{V^\ast,V} \diff{s}.
  \end{align*}
  Therefore, altogether we obtain the following representation
  \begin{align}
    \label{eq6:reptest}
    \begin{split}
      \ska[H]{E^n - E^{n-1} }{\Theta^n}
      & = - k \big\langle \A(\tn_n) U^n_h - \A(\tn_n) u(t_n), \Theta^n
      \rangle_{V^\ast,V}\\
      &\qquad - k \big\langle \A(\tn_n) u(t_n) - \A(\tn_n) u(\tn_n),
      \Theta^n \big\rangle_{V^\ast,V}\\
      &\qquad - \int_{t_{n-1}}^{t_n} \big\langle
      \A(\tn_n) u(\tn_n) - \A(s) u(s), \Theta^n \big\rangle_{V^\ast,V}
      \diff{s}\\
      &\qquad + \int_{t_{n-1}}^{t_{n}} \big( f(\tn_n) - f(s), \Theta^n \big)_H
      \diff{s}\\
      &=: \Gamma_1 + \Gamma_2 +\Gamma_3 +\Gamma_4.
    \end{split}
  \end{align}
  We give estimates for the four terms $\Gamma_i$, $i \in \{1,\ldots,4\}$, in
  \eqref{eq6:reptest} separately. By recalling $\Theta^n = P_h E^n = U_h^n -
  P_h u(t_n)$ 
  the first term is estimated using 
  Assumption~\ref{as:nonlinPDE}~\ref{item:Alip}
	and~\ref{item:Amon} as follows:
  \begin{align}
    \label{eq6:test1}
    \begin{split}
      \Gamma_1 &= - k \big\langle \A(\tn_n) U^n_h - \A(\tn_n) P_h u(t_n),
      \Theta^n 
      \big\rangle_{V^\ast,V}\\
      &\qquad - k \big\langle \A(\tn_n) P_h u(t_n) - \A(\tn_n) u(t_n),
      \Theta^n \big \rangle_{V^\ast,V}\\
      &\le - k \mu \big\| \Theta^n \big\|^2_V + 
      k L \big\| (I - P_h) u(t_n) \big\|_V \big\| \Theta^n \big\|_V\\
      &\le - k \mu \big\| \Theta^n \big\|^2_V + 
      k \frac{L^2}{\mu} \big\| \Xi^n \big\|_V^2 + k \frac{\mu}{4}
      \big\| \Theta^n \big\|^2_V.
    \end{split}
  \end{align}
  Observe that we also applied the weighted Young inequality in the last step.

  We similarly obtain an estimate for the second summand in
  \eqref{eq6:reptest} of the form
  \begin{align*}
    &\Gamma_2 \le k \big\| \A(\tn_n) u(t_n) - \A(\tn_n) u(\tn_n) 
    \big\|_{V^\ast}
    \big\| \Theta^n \big\|_V \\
    &\quad \le k \frac{L^2}{\mu} \| u(t_n) - u(\tn_n) \|_V^2 
    + k \frac{\mu}{4} \big\| \Theta^n \big\|_V^2.
  \end{align*}
  Since $|t_n - \xi_n(\omega)| \le k$ for every $\omega \in \Omega$
  and $u \in C^\gamma( [0,T];V)$ we therefore
  conclude 
  \begin{align}
    \label{eq6:test2}
    \Gamma_2 \le k^{1 + 2 \gamma} \frac{L^2}{\mu} \| u
    \|_{C^\gamma( [0,T];V)}^2 
    + k \frac{\mu}{4} \big\| \Theta^n \big\|_V^2.
  \end{align}

  Concerning the term $\Gamma_3$ in \eqref{eq6:reptest}, let us recall that 
  both $\Theta^n$ and $\xi_n$ are square-integrable
  random variables which are $\F_n$-measurable. Moreover, $\Theta^n$ takes
  values in $V_h \subset V$, while $\omega \mapsto
  \A(\xi_n(\omega))u(\xi_n(\omega))$ takes almost surely
  values in $H$ due to \eqref{eq6:L2cond}. 
  Therefore, after taking expectation we obtain 
  \begin{align}
    \label{eq6:gamma3}
    \begin{split}
      \E\big[ \Gamma_3 \big]
      &= \E \Big[ \int_{t_{n-1}}^{t_{n}} 
      \big\langle \A(\tn_n) u(\tn_n) - \A(s) u(s), \Theta^n
      \big\rangle_{V^\ast,V} \diff{s} \Big]\\
      &= \E \Big[ \int_{t_{n-1}}^{t_{n}} 
      \big( \A(\tn_n) u(\tn_n) - \A(s) u(s), \Theta^n \big)_H
      \diff{s} \Big]\\
      &=\E \Big[ \int_{t_{n-1}}^{t_{n}} 
      \big( \A(\tn_n) u(\tn_n) - \A(s) u(s), \Theta^n - \Theta^{n-1}
      \big)_H \diff{s} \Big]\\
      &\quad + \E \Big[ \int_{t_{n-1}}^{t_{n}} 
      \big ( \A(\tn_n) u(\tn_n) - \A(s) u(s), \Theta^{n-1}
      \big)_H \diff{s} \Big].
    \end{split}
  \end{align}
  Standard arguments then directly yield 
  a bound for the first summand of the form
  \begin{align*}
    &\E \Big[ \int_{t_{n-1}}^{t_{n}} 
    \big(  \A(\tn_n) u(\tn_n) - \A(s) u(s), \Theta^n - \Theta^{n-1}
    \big)_H \diff{s} \Big]  \\
    &\quad \le \Big( \E \Big[ \int_{t_{n-1}}^{t_{n}}
    \| \A(\tn_n) u(\tn_n) - \A(s) u(s) \|^2_H \diff{s} \Big]
    \Big)^{\frac{1}{2}} \big( k \E \big[ \| \Theta^n - \Theta^{n-1} \|^2_H 
    \big]
    \big)^{\frac{1}{2}}\\
    &\quad \le k  \E \Big[ \int_{t_{n-1}}^{t_{n}}
    \| \A(\tn_n) u(\tn_n) - \A(s) u(s) \|^2_H \diff{s} \Big] + 
    \frac{1}{4} \E \big[ \| \Theta^n - \Theta^{n-1} \|^2_H\big]\\
    &\quad \le 4 k \int_{t_{n-1}}^{t_{n}} \| \A(s) u(s) \|^2_H \diff{s}
    + \frac{1}{4} \E \big[ \| \Theta^n - \Theta^{n-1} \|^2_H\big],
  \end{align*}
  where the last step follows from 
  \begin{align*}
    \E\big[ k \| \A(\tn_n) u(\tn_n)\|^2_H \big] = \int_{t_{n-1}}^{t_{n}} \|
    \A(s) u(s) \|^2_H \diff{s}.
  \end{align*}
  To estimate the second summand in \eqref{eq6:gamma3}, we make use of
  the tower property for conditional expectations and the fact that
  $\Theta^{n-1}$ is $\F_{n-1}$-measurable. This yields
  \begin{align*}
    &\E \Big[ \int_{t_{n-1}}^{t_{n}} 
    \big ( \A(\tn_n) u(\tn_n) - \A(s) u(s), \Theta^{n-1}
    \big)_H \diff{s} \Big] \\
    &\quad = 
    \E \Big[ \int_{t_{n-1}}^{t_{n}} \E \big[ 
    \big ( \A(\tn_n) u(\tn_n) - \A(s) u(s), \Theta^{n-1}
    \big)_H \big| \F_{n-1} \big] \diff{s} \Big] \\
    &\quad = \E \Big[ k \E \big[ 
    \big ( \A(\tn_n) u(\tn_n), \Theta^{n-1}
    \big)_H \big| \F_{n-1} \big] 
    - \int_{t_{n-1}}^{t_{n}} \big( \A(s) u(s), \Theta^{n-1}
    \big)_H \diff{s} \Big] = 0,
  \end{align*}
  where the last step follows from
  \begin{align*}
    k \E \big[ \big ( \A(\tn_n) u(\tn_n), \Theta^{n-1} \big)_H \big| \F_{n-1}
    \big] 
    &= \big( k \E \big[ \A(\tn_n) u(\tn_n) \big| \F_{n-1}\big], \Theta^{n-1}
    \big)_H \\
    &= \big( k \E \big[ \A(\tn_n) u(\tn_n) \big], \Theta^{n-1}
    \big)_H\\
    &= \int_{t_{n-1}}^{t_n} \big( \A(s) u(s), \Theta^{n-1} \big)_H \diff{s}
  \end{align*}
  due to the independence of $\xi_n$ from $\F_{n-1}$. Altogether this shows
  \begin{align}
    \label{eq6:test3}
    \E\big[ \Gamma_3 \big] \le 
    4 k \int_{t_{n-1}}^{t_{n}} \| \A(s) u(s) \|^2_H \diff{s}
    + \frac{1}{4} \E \big[ \| \Theta^n - \Theta^{n-1} \|^2_H\big].    
  \end{align}
  
  The same steps with $f(\cdot)$ in place of $\A(\cdot) u(\cdot)$ also 
  yield an estimate for $\Gamma_4$. Therefore,
  \begin{align}
    \label{eq6:test4}
    \E\big[ \Gamma_4 \big] \le 
    4 k \int_{t_{n-1}}^{t_{n}} \| f(s) \|^2_H \diff{s}
    + \frac{1}{4} \E \big[ \| \Theta^n - \Theta^{n-1} \|^2_H\big].    
  \end{align}
  In summary, after taking expectation and inserting \eqref{eq6:test1},
  \eqref{eq6:test2}, \eqref{eq6:test3}, and  \eqref{eq6:test4} into
  \eqref{eq6:polar} we obtain
  \begin{align*}
    &\frac{1}{2} \E \big[ \|\Theta^n \|^2_{H} - \|\Theta^{n-1} \|^2_{H} 
    + \|\Theta^n - \Theta^{n-1} \|^2_{H} \big] \\
    &\quad = \E \big[ \Gamma_1 + \Gamma_2 +\Gamma_3 +\Gamma_4 \big]\\
    &\quad \le - \frac{1}{2} k  \mu \E \big[ \big\| \Theta^n \big\|^2_V \big]
    + k \frac{L^2}{\mu} \big\| \Xi^n \big\|_V^2 
    + k^{1 + 2 \gamma} \frac{L^2}{\mu} \| u
    \|_{C^\gamma( [0,T];V)}^2\\
    &\qquad + 4 k \int_{t_{n-1}}^{t_{n}} \| \A(s) u(s) \|^2_H \diff{s}
    +4 k \int_{t_{n-1}}^{t_{n}} \| f(s) \|^2_H \diff{s}\\
    &\qquad+  \frac{1}{2} \E \big[ \| \Theta^n - \Theta^{n-1} \|^2_H\big].
  \end{align*}
  
  After canceling the last term from both sides of the inequality, we sum over
  $n \in \{1,\ldots, j\}$ for some arbitrary $j \in \{1,\ldots,N\}$. Moreover,
  since $U_h^0 = P_h u_0$ we also have $\Theta^0 = 0$. Hence
  we obtain 
  \begin{align*}
    &\E \big[ \|\Theta^j \|^2_{H} \big] + \frac{1}{2} k \mu \sum_{n = 1}^j
    \E \big[ \|\Theta^n \|^2_{V} \big]\\
    &\quad \le k \frac{L^2}{\mu}\sum_{n = 1}^j \big\| \Xi^j \big\|_V^2
    + k^{2 \gamma} T \frac{L^2}{\mu} \| u \|_{C^\gamma( [0,T];V)}^2
    \\
    & \qquad + 4 k \big( \| \A(\cdot) u(\cdot) \|^2_{L^2(0,T;H)} + 
    \| f \|_{L^2(0,T;H)}^2 \big).  
  \end{align*}
  The proof is completed by taking the maximum over $j \in \{1,\ldots,N\}$ and
  an application of \eqref{eq6:estXIL2}.
\end{proof}

\begin{remark}
	Let us briefly discuss the additional regularity conditions
	\eqref{eq6:Hoeldercond} and \eqref{eq6:L2cond} in Theorem~\ref{th6:conv}.
	First note that since $f \in L^2(0,T;H)$ the condition \eqref{eq6:L2cond} is
	essentially equivalent to 
	\begin{align*}
		\dot{u} \in L^2(0,T;H).
	\end{align*}
	A sufficient condition for \eqref{eq6:Hoeldercond} is then to additionally
	require
	\begin{align*}
		\dot{u} = f - \A(\cdot) u(\cdot) \in L^{q}(0,T;V)
	\end{align*}
	with $q = \frac{1}{1-\gamma}$. In Section~\ref{sec:PDEreg} we will discuss
  more explicit classes of linear and semilinear evolution equations,
  whose solutions enjoy the required regularity. 
\end{remark}

\section{Regularity of non-autonomous evolution equations}
\label{sec:PDEreg}
To prove a rate of convergence in Section~\ref{sec:nonlinPDE}, 
we had to to impose additional assumptions on the regularity of the exact
solution $u$. In the following, we will discuss cases where this particular
regularity can be expected.

We begin by considering a class of linear problems that fulfills the 
regularity conditions imposed in Theorem~\ref{th6:conv}. 
As in Section~\ref{sec:nonlinPDE},
we consider the real, separable 
Hilbert spaces $(V,\ska[V]{\cdot}{\cdot}, \|\cdot\|_V)$ and 
$(H,\ska[H]{\cdot}{\cdot}, \|\cdot\|_H)$ that form the Gelfand triple
\begin{align*}
  V \stackrel{c,d}{\hookrightarrow} H \cong H^* 
  \stackrel{c,d}{\hookrightarrow} V^*.
\end{align*}
Further, we state the following assumption to obtain a suitable evolution 
operator.

\begin{assumption}
	\label{as:linPDE}
  For all $t \in [0,T]$ let $a_0(t;\cdot,\cdot)\colon V\times V \to \R$ be
  a	bilinear form that fulfills the following conditions:
	\begin{enumerate}[label=(\roman*)]
		\item For every $v_1, v_2 \in V$ the mapping $a_0(\cdot; 
		v_1,v_2)\colon[0,T]\to \R$
		is measurable.
		\item \label{item:aLinbounded} 
		There exists $\beta \in (0,\infty)$ such that for all $t \in [0,T]$
		it holds true that
		\begin{align*}
  		|a_0(t;v_1,v_2)| \le \beta \|v_1\|_V \| v_2 \|_V , \quad
  		\text{ for all } v_1, v_2 \in V.
		\end{align*}
		\item \label{item:Apos}
		There exists $\mu \in (0,\infty)$ such that for all $t \in
		[0,T]$ it holds true that
		\begin{align*}
  		a_0(t;v,v) \ge \mu \| v \|^2_V, \quad \text{ for all } v \in V.         
  		\end{align*}
	\end{enumerate}
\end{assumption}

For every $t \in [0,T]$ let $\A_0(t) \colon V\to V^*$ and $A_0(t) \colon
\dom(A_0(t)) \subset  H \to H$ denote the associated operators to the bilinear form 
$a_0(t;\cdot,\cdot)$ from Assumption~\ref{as:linPDE}. More precisely, 
the linear operator $\A_0(t)$ is uniquely determined by
\begin{align*}
  \big\langle \A_0(t) v_1, v_2 \big\rangle_{V^\ast,V} &= a_0(t; v_1, v_2), 
  \quad
  \text{ for all } v_1,v_2 \in V.
\end{align*}
Moreover, we set $\dom(A_0(t)) := \{ v \in V \, : \, \A_0(t) v \in H \}$ and define
$A_0(t)$ as the restriction of $\A_0(t)$ to the domain $\dom(A_0(t))$.
Note that $\dom(A_0(t))$ becomes a Banach space if endowed with the graph norm.
For given $u_0 \in V$ and $f \in L^2(0,T;H)$ we consider
the following linear and non-autonomous problem
\begin{equation}\label{eq7:linPDE}
	\begin{split}
    \begin{cases}
      \dot{u}(t)+\A_0(t)u(t)=f(t),& \quad t\in (0,T],\\
      u(0) = u_0. &      
    \end{cases}
  \end{split}
\end{equation}
For the formulation of the regularity result we first recall that the initial
value problem \eqref{eq7:linPDE} 
is said to have \emph{maximal} $L^p$-\emph{regularity in} $H$ 
for some $p \in [2,\infty)$ if for all $f \in L^p(0,T;H)$ we have a unique
solution $u \in \mathcal{W}(0,T)$ with $\dot{u} \in L^p(0,T;H)$ and
$\A_0(\cdot)u(\cdot) \in L^p(0,T;H)$.

\begin{theorem} 
  \label{th:regularityLin}
  For every $t \in [0,T]$ let $a_0(t;\cdot,\cdot)\colon V\times V \to \R$ 
  fulfill Assumption~\ref{as:linPDE}.  
	Suppose there exists $M \geq 0$ such that 
	$\|A_0(0)A_0^{-1}(t)\|_{\L(H)}\leq M$ holds true
	for all $t\in [0,T]$, that the embedding 
	$\dom(A_0(0))\hookrightarrow H$ is compact, and that
	$\dom(A_0(0)^{\frac12})=V$ (Kato's square root property).  
	Fix $u_0\in V$ and assume that the Cauchy problem 
	\eqref{eq7:linPDE} has maximal $L^p$-regularity in $H$ for some $p \in
  (2,\infty)$. 
	Then $u\in C^{\gamma}([0,T];V)$ for $\gamma <\frac12-\frac1p$.
\end{theorem}

\begin{proof}
	Since \eqref{eq7:linPDE} has maximal $L^p$-regularity it follows that
	$u\in W^{1,p}(0,T;H)$ and that
	$\int_0^T\|A_0(t)u(t)\|_H^p\, \diff{t} <\infty$. As 
  $\|A_0(0) A^{-1}_0(t)\|_{\mathcal{L}(H)} \leq M$
	for all $t\in [0,T]$, we have that $u\in 
	L^p(0,T;\dom(A_0(0)))$.  Let $0<\epsilon<\frac12-\frac1p$ and consider 
	the continuous embeddings of real interpolation spaces (for the precise 
	definition of an interpolation space see, for example 
	\cite[Sec.~1.3.2.]{triebe1978})
	\begin{align*}
		(H,\dom(A_0(0)))_{\frac12+\epsilon,p}
		\hookrightarrow (H,\dom(A_0(0)))_{\frac12,1}
		\hookrightarrow \dom(A_0(0)^{\frac12})=V,
	\end{align*}
	where a proof for the first embedding can be found in, for example, 
	\cite[Sec.~1.3.3.]{triebe1978} and in \cite[Sec.~1.15.2.]{triebe1978} 
	for the second.
	Then, it follows from \cite[Thm~5.2]{amann2000} that for all
	$0\leq \gamma <\frac12-\epsilon-\frac1p$ we have $u\in C^{\gamma}([0,T];V)$ 
	and the proof is complete.
\end{proof}

A sufficient condition for maximal $L^p$-regularity for 
\eqref{eq7:linPDE} is found in \cite{haak2015}: If
$a_0(t;\cdot,\cdot)\colon V\times V\mapsto \R$ fulfills 
Assumption~\ref{as:linPDE} and 
\begin{align*}
	|a_0(s;u,v)-a_0(t,u,v)|\leq \omega(|s-t|) \|u\|_V\|v\|_V, \quad \text{ for
  all } u, v \in V \text{ and } s,t \in [0,T]
\end{align*}
where $\omega \colon [0,T]\to [0,\infty)$ is a non-decreasing function such 
that 
\begin{align}
  \label{eq7:omegacond}
	\int_0^T\left(\frac{\omega (t)}{t}\right)^p\, \diff{t} <\infty  
\end{align}
for some $p \in (2,\infty)$ and $u_0\in (H,\dom(A_0(0)))_{1-\frac1p, p}$, then
the initial value problem \eqref{eq7:linPDE} has maximal
$L^p$-regularity in $H$.

\begin{example}
  \label{example1}
  Let $\D \subset \mathbb{R}^d$ be a bounded domain with either a smooth 
  boundary or a polygonal boundary if $\D$ is also convex. 
  We set $H=L^2(\D)$ and $V=H^1_0(\D)$.  Then we consider
  the bilinear form $a_0(t;\cdot,\cdot)\colon V\times V\mapsto \R$ given by
  \begin{align*}
    a_0(t;u,v)=\int_{\D} \alpha(t,x) \nabla u(x)\cdot\nabla v(x) \diff{x}.
  \end{align*}
  The coefficient function $\alpha \colon [0,T] \times \D \to \R$ 
  is assumed to satisfy uniform bounds
\begin{align*}
	0<\alpha_1\le \alpha(t;x) \le \alpha_2 ,  
  \quad \text{ for all } t,x\in[0,T]\times \D. 
\end{align*}
In addition, we assume that $\alpha(t, \cdot)$ is sufficiently
smooth with respect to $x$ for every $t \in [0,T]$. Further, 
we suppose that
\begin{align*}
	\| \alpha(t,\cdot)-\alpha(s,\cdot)\|_{L^{\infty}(\D)}\le
	\omega(|t-s|) 
\end{align*}
with a mapping $\omega \colon [0,T] \to [0,\infty)$ as in 
\eqref{eq7:omegacond}.  Then
$a_0(t;\cdot,\cdot)$ satisfies Assumption~\ref{as:linPDE}.
Moreover, the domain $\dom(A_0(t))=H^2(\D)\cap H^1_0(\D)$ is constant in
$t \in [0,T]$ and we have maximal $L^p$-regularity for any $p \in (2,\infty)$ 
with \eqref{eq7:omegacond}, see \cite{haak2015}. 
Theorem~\ref{th:regularityLin} then yields
that $u \in C^{\gamma}([0,T];V)$ for $\gamma <\frac12-\frac1p$ and all
regularity conditions of Theorem~\ref{th6:conv} are satisfied.
A sufficient condition for the initial value is, for example, to choose $u_0 
\in H^2(\D)\cap H^1_0(\D)$.
\end{example}

\begin{remark}
	For further regularity results of linear problems see 
  \cite[Chap.~II.1]{amann1995} or \cite[Chap.~6]{lunardi1995}. There,
  non-autonomous linear evolution equations are considered and the existence of
  a H\"older continuous solution is proved.
\end{remark}

The following assumption states sufficient conditions on a nonlinear
perturbation of $\A_0$ such that the regularity results can be extended from
the linear case to a class of semilinear problems. 

\begin{assumption} \label{as:perturbation}
  Let $\B \colon [0,T] \times H  \to H $ fulfill the following conditions:
	\begin{itemize}
		\item[(i)] For every $v_1, v_2 \in H$ the mapping $\ska[H]{ \B(\cdot) v_1 
		}{v_2} \colon [0,T] \to \R$ is measurable.
		\item[(ii)] There exists a constant $M \geq 0$ such that 
		$\| \B(t) 0 \|_{H}	\leq M$ for every $t\in [0,T]$.
		\item[(iii)] For every $t \in [0,T]$ and $v_1,v_2 \in H$ it holds true 
		that 
			\begin{align*}
				\ska[H]{ \B( t ) v_1 - \B(t ) v_2}{v_1- v_2} \geq 0.
			\end{align*}
		\item[(iv)] There exists $L >0$ such that 
			\begin{align*}
				\|\B(t) v_1 - \B( t ) v_2 \|_{H} \leq L \|v_1 - v_2 \|_H
			\end{align*}
			holds for every $t\in [0,T]$ and $v_1,v_2 \in H$.
	\end{itemize}
\end{assumption}

\begin{remark}
  Applying Remark~\ref{rem:Gaarding}, we can weaken 
  Assumption~\ref{as:perturbation}~$(iii)$ to 
  \begin{itemize}
    \item[$(iii')$] There exists $\kappa \in [0,\infty)$ such that for every $t \in 
    [0,T]$ and $v_1,v_2 \in H$ it holds true that 
    \begin{align*}
      \ska[H]{ \B( t ) v_1 - \B(t ) v_2}{v_1- v_2} \geq -\kappa \|v_1 - v_2\|_H^2.
    \end{align*}
  \end{itemize}
  at the cost that the constants in the error estimate grow exponentially 
  with growing $T$.
\end{remark}

\begin{example}
  \label{example2}
  In the situation of Example~\ref{example1} let
  $b \colon [0,T] \times \D \times \R \to \R$ be a mapping satisfying
  the following conditions:
  \begin{itemize}
    \item[(i)] For every $z \in \R$ the mapping $b(\cdot,\cdot,z) \colon [0,T]
      \times \D \to \R$ is measurable.
    \item[(ii)] There exists a constant $m \in [0,\infty)$ such that
      $|b(t,x,0)| \le m$ for every $t \in [0,T]$, $x \in \D$.
    \item[(iii)] For every $t \in [0,T]$ and $x \in \D$ the mapping
      $b(t,x,\cdot) \colon \R \to \R$ is non-decreasing and globally Lipschitz
      continuous. 
  \end{itemize}
  Then the Nemytskii operator $\B \colon [0,T] \times L^2(\D) \to L^2(\D)$
  defined by $(t,v) \mapsto b(t,\cdot,v(\cdot))$ satisfies
  Assumption~\ref{as:perturbation}. 
\end{example}

Assumptions~\ref{as:linPDE} and~\ref{as:perturbation} in mind, we now consider 
the nonlinear problem
\begin{equation}\label{eq7:linPDEperturbated}
	\begin{split}
    \begin{cases}
      \dot{u}(t)+\A_0(t)u(t) + \B(t) u(t) = f(t),& \quad t\in (0,T],\\
		  u(0)=u_0.&
    \end{cases}
	\end{split}
\end{equation}
A simple insertion of Assumptions~\ref{as:linPDE} and~\ref{as:perturbation} 
proves that the sum of the operators $\A = \A_0 + \B$ fulfills 
Assumption~\ref{as:nonlinPDE}. Further, we obtain the same regularity result 
for the perturbed problem \eqref{eq7:linPDEperturbated} as for the linear 
problem \eqref{eq7:linPDE}.

\begin{theorem} \label{th:regularityLinPerturbed}
	Let the assumptions of Theorem~\ref{th:regularityLin} be satisfied for some
  $p \in (2,\infty)$ and let 
	$\B$ fulfill Assumption~\ref{as:perturbation}. Then the solution $u$ to 
	\eqref{eq7:linPDEperturbated} belongs to $C^{\gamma}([0,T];V)$ for
  $\gamma <\frac12-\frac1p$.
\end{theorem} 

\begin{proof}
	The proof for the regularity follows a similar idea as presented in 
	\cite[Thm~2.9]{meidnerVexler2017}.
	To this end let $u \in C([0,T];H)$ be the unique solution of 
	\eqref{eq7:linPDEperturbated}, see Proposition \ref{prop:exPDE}. We consider 
	the function $g = f - \B u$ which fulfills 
	\begin{align*}
		\|g \|_{L^p(0,T;H)} 
		&\leq \|f \|_{L^p(0,T;H)} + \| \B(\cdot) u(\cdot) - \B(\cdot) 0 
		\|_{L^p(0,T;H)} 
		+ \|\B(\cdot) 0\|_{L^p(0,T;H)} \\
		&\leq \|f \|_{L^p(0,T;H)} + L \| u \|_{L^p(0,T;H)} + T^{\frac{1}{p}} M.
	\end{align*}
	Thus $g \in L^p(0,T;H)$ and the solution $v$ of the linear problem 
	\begin{equation}\label{eq7:perturbatedHom}
		\begin{split}
			\begin{cases}
				\dot{v}(t)+\A_0(t)v(t)  = g(t),& \quad t\in (0,T],\\
				v(0)= u_0.
			\end{cases}
		\end{split}
	\end{equation}
	is an element of the space $C^{\gamma}([0,T];V)$ due to Theorem 
	\ref{th:regularityLin}.
	
	This now enables us to apply a bootstrap argument for the regularity of the 
	solution $u$ of \eqref{eq7:linPDEperturbated}. 
	Both \eqref{eq7:linPDEperturbated} and \eqref{eq7:perturbatedHom} are 
	uniquely solvable and an insertion of $u$ in \eqref{eq7:perturbatedHom} 
	shows that $u$ also solves the linear initial value problem. 
	Therefore, $u= v$ holds and we obtain that $ u \in C^{\gamma}([0,T];V)$.
\end{proof}

\begin{remark}
  The verification of the regularity conditions in Theorem~\ref{th6:conv}
  for general nonlinear PDEs can be quite challenging. However, besides the
  linear and semilinear problems discussed in this section, there are further 
  classes of nonlinear problems that yield H\"older continuous solutions.
  For more general regularity results of semilinear problems we
  refer the reader to \cite[Chap.~7]{lunardi1995}. 
	In \cite{ArendtDuelli2006} some quasi-linear problems are considered. 
	They prove maximal $p$-regularity for these problems, which could 
	potentially be extended to fit our setting as well. 
	A further class of nonlinear problems is considered in 
	\cite{ostermann2002}, where regularity results from 
	\cite{lunardi1995} are used. 
	Here, a rather strong temporal regularity condition is imposed
  on the coefficients which would also lead to higher order convergence results
  of the classical backward Euler method. But, 
  as it can be seen from our numerical examples in Section~\ref{sec:numexpODE}
  and Section~\ref{sec:PDEnum}, the randomized schemes
  \eqref{eq4:RandBackEuler} and \eqref{eq6:parabolicScheme} might still offer
  more reliable results in comparison to their deterministic counterparts if, 
  for instance, the coefficients are smooth but highly oscillating.  
\end{remark}

\section{Numerical experiment with a non-autonomous PDE}
\label{sec:PDEnum}

In this section we finally illustrate the usability of the randomized backward 
Euler method \eqref{eq6:parabolicScheme} for the numerical solution of 
evolution equations. To this end, we follow a similar approach as for ODEs 
presented in Section~\ref{sec:numexpODE}. Here, we consider a nonlinear 
PDE of the form 
\begin{align}\label{eq:7PDEnum}
  \begin{cases}
    u_{t}(t,x) - u_{xx}(t,x) + b(u(t,x)) = f(t,x), &
    \quad (t,x) \in (0,1)^2,\\
    u(t,0) = u(t,1) = 0,& \quad t \in (0,1),\\
    u(0,x) = u_0(x),& \quad x\in (0,1),
  \end{cases}
\end{align}
where we choose the function $b$ given by
\begin{align} \label{eq8:defB}
  b \colon \R \to \R, \quad b(x) = 
  \begin{cases}
    |x|^{\tilde{p}-2}x, \quad &\text{for } |x|\leq R,\\
    R^{\tilde{p}-2}x, &\text{for } |x| > R,
  \end{cases}
\end{align}
for a fixed $R \in (0,\infty)$ and $\tilde{p} \in [2,\infty)$ as well as suitable 
functions $f$ and $u_0$ which are specified further below.
Using $H = L^2(0,1)$ and $V = H_0^1(0,1)$, \eqref{eq:7PDEnum} fits 
into the setting of Section~\ref{sec:PDEreg}, where
\begin{align*}
  a_0(t; v_1, v_2)
  = \inner[V\times V^*]{\A_0 v_1}{v_2} 
  = \int_{\Omega} v_1'(x) v_2'(x) \diff{x}, \quad\text{for all } v_1,v_2 \in V,
\end{align*}
fulfills Assumption~\ref{as:linPDE} and
\begin{align*}
  \ska[H]{\B v_1}{v_2} = \int_{\Omega} b(v_1(x)) v_2(x) \diff{x}, 
  \quad \text{for all } v_1,v_2 \in H,
\end{align*}
fulfills Assumption~\ref{as:perturbation}. For a function $u_0 \in V$ and $f 
\in L^p(0,1;H)$ with $p \in [2,\infty)$ \eqref{eq:7PDEnum} has maximal 
$L^p$-regularity. Furthermore, the assumptions of 
Theorem~\ref{th:regularityLinPerturbed} are fulfilled such that the solution 
$u$ is an element of $C^{\gamma}([0,T];V)$ for $\gamma < \frac{1}{2} - 
\frac{1}{p}$.

In our numerical example we consider a highly oscillating function $w$.
For $P = 2^{-K}$, $K\in \N$, $w$ is the continuous, piecewise linear 
function determined by
\begin{align*}
  w( i P) = 
  \begin{cases}
    iP^2, \quad &\text{ for } i \in \{ 0,\dots, 2^K \} \text{ odd},\\
    0, \quad &\text{ for } i \in \{ 0,\dots, 2^K \} \text{ even},
  \end{cases}
\end{align*}
and the affine linear interpolation of these values for all other $t \in
(0,1)$. The function $w$ is then weakly differentiable with derivative 
$\dot{w}$ given by
\begin{align*}
  \dot{w}(t) = 
    \begin{cases}
      iP, \quad &\text{ for } t \in [(i-1)P,i P ), \ i\in \{ 1,\dots, 2^K \} 
      \text{ odd,}\\
      -(i-1)P, \quad &\text{ for } t \in [(i-1)P,iP ),  \ i\in \{ 1,\dots, 2^K \} 
      \text{ even}.
  \end{cases}
\end{align*}
For the functions
\begin{align*}
  f(t,x) &= (x^2 - x^3)\dot{w}(t) - (2 - 6x)w(t) + \sin(\pi x) + b((x^2 - x^3)w(t) 
  + \pi^{-2} \sin(\pi x)),\\
  u_0(x) &= \pi^{-2} \sin(\pi x)
\end{align*}
the solution is given by
\begin{align*}
  u(t,x) = (x^2 - x^3)w(t) + \pi^{-2} \sin(\pi x)
\end{align*}
as can be seen by a simple insertion.

Then we see that $f \in L^{\infty}(0,1;H)$, $u_0 \in V$, $u \in 
C^{\gamma}([0,1];V)$ for every $\gamma \in (0,1)$ and \eqref{eq6:L2cond} 
holds true, thus the assumptions of Theorem~\ref{th6:conv} are fulfilled.

\begin{figure}[ht] 
  \centering
  \includegraphics[width=0.7\textwidth]{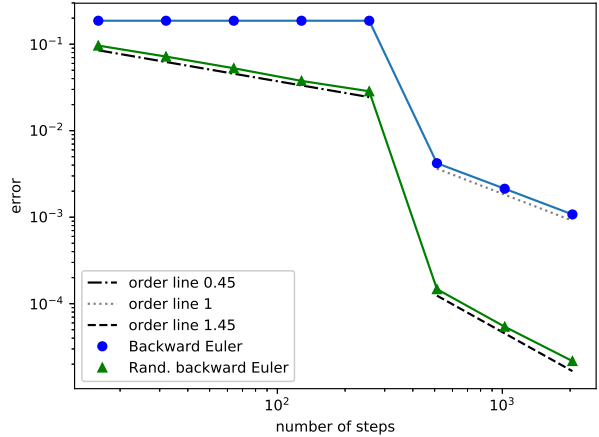}
  \caption{$L^2(\Omega;L^2(0,1))$-errors of the classical backward Euler 
  method
    and scheme \eqref{eq6:parabolicScheme} for equation 
    \eqref{eq:7PDEnum}.}
  \label{fig2}
\end{figure}

The numerical behavior of this problem is very similar to the ODE example 
in Section~\ref{sec:numexpODE}. The right-hand side $f$ is highly 
oscillating. Thus, the classical backward Euler method needs a step size 
smaller than $P$ in order to give an accurate numerical approximation. The 
randomized scheme \eqref{eq6:parabolicScheme}, on the other hand, yields 
much better approximations of the solution for larger values of the step 
size.

In our numerical test displayed in Figure~\ref{fig2}, we considered $R = 10$ 
and $\tilde{p} = 4$ in \eqref{eq8:defB}, $P = 2^{-9}$ and step sizes $k = 
2^n$ with $n \in \{ 4, \dots, 11 \}$. To approximate the 
$L^2(\Omega;L^2(0,1))$-norm of the error we used $200$ Monte Carlo 
iterations. Since we are only interested in demonstrating the temporal 
convergence, we use a fixed finite element space with $500$ degrees of 
freedom based on a uniform mesh in order to keep the spatial error on a 
negligible level for all considered temporal step sizes. For the 
implementation we used the finite element software package FEniCS 
\cite{fenics2012}.

The results are well comparable to the results for the ODE example 
in Section~\ref{sec:numexpODE}. When the step size is larger than the 
value $P$, we can recognize a convergence rate of $0.45$ for the 
randomized scheme. 
On the other hand, the error of the classical backward Euler method does 
not decrease for these step sizes. The errors of both schemes improve 
significantly when the step size is sufficiently small to resolve the 
oscillations. After that we see the classical rate of $1$ for the deterministic 
scheme and a rate of $1.45$ in our randomized scheme.

\section*{Acknowledgment}

The authors like to thank Wolf-J\"urgen Beyn for very helpful comments on
non-autonomous evolution equations and Rico Weiske for good advise 
on programming. Also we like to thank two anonymous referees for 
their valuable suggestions.

This research was partially carried out in the framework of \textsc{Matheon}
supported by Einstein Foundation Berlin.  ME would like to thank the Berlin
Mathematical School for the financial support. RK also gratefully acknowledges
financial support by the German Research Foundation (DFG) through the research 
unit FOR 2402 -- Rough paths, stochastic partial differential equations and
related topics -- at TU Berlin. 

\def\cprime{$'$} \def\polhk#1{\setbox0=\hbox{#1}{\ooalign{\hidewidth
			\lower1.5ex\hbox{`}\hidewidth\crcr\unhbox0}}}

\end{document}